\documentclass[english]{jnsao}
\usepackage[utf8]{inputenc}
\usepackage{lineno}

\usepackage{graphicx}
\usepackage{epstopdf}
\usepackage{algorithmic}
\usepackage{amsopn}
\usepackage[noadjust]{cite}
\usepackage{mathtools}
\usepackage[capitalize,nameinlink]{cleveref}

\usepackage{nicefrac}

\usepackage{subcaption} 
\usepackage{wrapfig}
\usepackage{float}

\usepackage{enumitem}
\setlist[enumerate]{wide=5mm, leftmargin=10mm, rightmargin=5mm, listparindent=5mm, parsep=0mm,
itemsep=2mm,
topsep=4mm}
\usepackage{ulem}
\ifpdf
\hypersetup{
  pdftitle={Optimal control of plasticity with inertia},
  pdfauthor={S.~Walther}
}
\fi

\title{Optimal control of plasticity with inertia}

\author{
Stephan Walther\thanks{TU Dortmund, Faculty of Mathematics,
Vogelpothsweg 87, 44227 Dortmund, Germany 
  (\email{stephan.walther@tu-dortmund.de},
  \url{http://www.mathematik.tu-dortmund.de/lsx}).}}
\shortauthor{Walther}

\acknowledgements{
    This research was supported by the German Research Foundation (DFG) under grant 
	number~ME 3281/9-1 within the priority program Non-smooth and Complementarity-based
	Distributed Parameter Systems: Simulation and Hierarchical Optimization (SPP~1962).
}

\manuscriptcopyright{{\copyright} the authors}
\manuscriptlicense{CC-BY-SA 4.0}

\newcommand{\define}{\coloneqq}

\newcommand{\shortspace}{\text{ \ \ }}
\newcommand{\mediumspace}{\text{ \ \ \ \ }}
\newcommand{\largespace}{\text{ \ \ \ \ \ \ }}

\newcommand{\bdot}{\boldsymbol{.}}
\newcommand{\boverdot}[1]{\overset{\bdot}{#1}}
\newcommand{\bdoubledot}{\boldsymbol{..}}
\newcommand{\bdoubleoverdot}[1]{\overset{\bdoubledot}{#1}}

\newcommand{\N}{\mathbb{N}}

\newcommand{\R}{\mathbb{R}}
\newcommand{\Rd}{\R^{d}}
\newcommand{\Rdd}{\R^{d \times d}}
\newcommand{\Rdds}{\R^{d \times d}_s}
\renewcommand{\O}{\Omega}

\newcommand{\symnabla}{\nabla^{s}}

\newcommand{\dualpair}[3]{{\langle #1 , #2 \rangle}_{#3}}
\newcommand{\scalarproduct}[3]{\left( #1 , #2 \right)_{#3}}

\newcommand{\norm}[2]{\|#1\|_{#2}}
\newcommand{\bignorm}[2]{\left\lVert#1\right\rVert_{#2}}

\newcommand{\sequence}[2]{\{ #1_{#2} \}_{#2 \in \N}}

\renewcommand{\C}{\mathbb{C}}

\newcommand{\B}{\mathbb{B}}
\newcommand{\D}{\mathbb{D}}
\newcommand{\E}{\mathbb{E}}

\renewcommand{\AA}{\mathcal{A}}

\newcommand{\RR}{\mathcal{R}}
\newcommand{\ZZ}{\mathcal{Z}}

\newcommand{\HH}{\mathcal{H}}
\newcommand{\YY}{\mathcal{Y}}
\renewcommand{\ZZ}{\mathcal{Z}}
\newcommand{\XX}{\mathcal{X}}

\newcommand{\LL}{\mathcal{L}}

\newcommand{\KK}{\mathcal{K}}

\newcommand{\FF}{\mathcal{F}}

\renewcommand{\SS}{\mathcal{S}}
\newcommand{\TT}{\mathcal{T}}

\newcommand{\XXX}{\mathfrak{X}}
\newcommand{\QQQ}{\mathfrak{Q}}
\newcommand{\ZZZ}{\mathfrak{Z}}

\newcommand{\embed}{\hookrightarrow}

\renewcommand{\max}{\operatorname{max}}
\renewcommand{\div}{\operatorname{div}}

\DeclareMathOperator*{\argmin}{arg\,min}

\newtheorem{assumption}[theorem]{Assumption}
\newtheorem{notation and assumption}[theorem]{Notation and Assumption}

\manuscriptsubmitted{2021-02-05}
\manuscriptaccepted{2021-10-12}
\manuscriptvolume{2}
\manuscriptnumber{7156}
\manuscriptyear{2021}
\manuscriptdoi{10.46298/jnsao-2021-7156}

\begin{document}

\maketitle

\begin{abstract}
    The paper is concerned with an
    optimal control problem governed by
    the equations of elasto plasticity with
	linear kinematic hardening    
    and the inertia term at small strain.
    The objective is to optimize the
    displacement field and plastic strain
    by controlling volume forces.
    The idea given in
    \cite{groger}
    is used to transform the state equation
    into an
    evolution variational inequality (EVI)
    involving a certain maximal monotone operator.
    Results from
    \cite{paper_abstract}
    are then used to analyze the EVI.
    A regularization is obtained via the
    Yosida approximation of
    the maximal monotone operator,
    this approximation is smoothed
    further to derive optimality conditions
    for the smoothed optimal control problem.
\end{abstract}

\section{Introduction}
\label{sec:1}

We consider the following
optimal control problem governed
by the equations of
elasto plasticity with linear kinematic
hardening and the inertia term
at small strain:
\begin{equation}\label{eq:optimization_problem_inertia}
\left\{\quad
\begin{aligned}
	\min \shortspace
	&
	J(u, \boverdot{u},z,f)
	=
	\Psi(u,\boverdot{u}, z) + \frac{\alpha}{2}\norm{f}{\XXX_c}^2, \\
	\text{ s.t. }
	\shortspace
	&
	\rho\bdoubleoverdot{u}
	-
	\div \mathbb{C} (\symnabla u - z)
	=
	f,
	\\
	&\boverdot{z}
	\in
	A(\mathbb{C}\symnabla u - (\mathbb{C} + \mathbb{B} )z),
	\\
	&
	(u, \boverdot{u}, z)(0) = (u_0, v_0, z_0),
	\\
	&
	u \in H^1(H^1_D(\O;\Rd))
	\cap
	H^2(L^2(\O;\Rd)),
	\\
	&	
	z \in H^1(L^2(\O;\Rdds)),
	\\
	&
	f \in \XXX_c.
\end{aligned}\right.
\end{equation}
Herein,
$\Omega \subset \Rd$
is the body under consideration
with density
$\rho$,
where
$d \in \N$
is the dimension.
Its boundary is split into two
disjoints parts
$\Gamma_D$
and
$\Gamma_N$.
Furthermore,
$u : [0, T] \times \Omega \to \Rd$
is the displacement field
and
$z : [0, T] \times \Omega \to \Rdds$
the plastic strain.
The initial data
$(u_0, v_0, z_0)$
is given and fixed.
The volume force is given by
$f : [0, T] \times \O \to \Rd$.
The time derivative of, for instance, the plastic strain
is denoted by
$\boverdot{z}$
and the symmetric gradient by
$\symnabla = \nicefrac{1}{2}(\nabla + \nabla^\top)$.
Moreover,
$\C$
is the elasticity tensor and
$\B$
the hardening parameter.
The flow rule is represented by the
maximal monotone operator
$A$,
in
\cref{sec:examples}
below we will choose the
von-Mises flow rule.
The control space
$\XXX_c$
is a nonempty and closed subspace of
$H^1([0,T];L^2(\O;\Rd))$
and
$\alpha > 0$
a fixed Tikhonov parameter.
Note that higher time regularity
of the control
was also required when analyzing other
problems with a
non-smooth hyperbolic evolution structure,
e.g. in
\cite{yousept2017optimal}.
The precise definitions and
assumptions are presented in
\cref{sec:2}
below.
Note that the problem
\cref{eq:optimization_problem_inertia}
is formulated only with the
displacement and plastic strain as state
variables.
The stress is normally given by
$\sigma = \C(\symnabla u - z)$
(and can thus be easily integrated in
$\Psi$),
however,
we eliminated it in
\cref{eq:optimization_problem_inertia}
for convenience.
Additionally,
the Dirichlet displacement and
Neumann boundary forces
(see the definition of the
$\div$-operator
in
\cref{sec:2})
are set to
zero,
cf.
\cref{rem:Neumann_and_Dirichlet_inertia}
below.
Regarding a detailed description
and derivation of the plasticity model, 
we refer to
\cite{simohughes, OttosenRistinmaa2005:1, lubliner2008plasticity}.

Let us put our work into perspective.
Optimal control problems governed by
plasticity were consider in
\cite{HerzogMeyerWachsmuth2010:2, allaire, wachsmuth, Wac12, Wac15, Wac16, paper_abstract, paper_stress, paper_displacement}.
The articles
\cite{HerzogMeyerWachsmuth2010:2, allaire}
are concerned with the static case of
elasto plasticity,
for further articles of the static case
we refer to the references therein.
For the time dependent quasi-static case
we are only aware of
\cite{wachsmuth, Wac12, Wac15, Wac16, paper_abstract, paper_stress, paper_displacement}.
The articles
\cite{wachsmuth, Wac12, Wac15, Wac16}
and the application in
\cite{paper_abstract}
are devoted to the case of
elasto plasticity  with
linear kinematic hardening,
whereas
\cite{paper_stress, paper_displacement}
are concerned with the case of
perfect plasticity,
that is,
with no hardening.
In contrast,
we consider the case of elasto plasticity
with the inertia term
(and linear kinematic hardening),
that is,
the second time derivative
of the displacement
(the accelaration)
multiplied by the density
$\rho$
is present
in the balance of momentum.
Due to this inertia term,
the equations are physically more
reasonable than
the quasi-static case
(of course,
the solution to both
systems might not differ much
when the accelaration of the body
is small,
which is the reasoning when
neglecting the inertia term).
As said above,
the application in
\cite{paper_abstract}
investigates quasi-static
(homogenized) plasticity
with hardening.
This application is analyzed by
applying results
concerned with an abstract
optimal control problem
governed by an first-order
evolution variational inequality
(EVI)
involving a maximal monotone operator.
As we will see below,
the state equation in
\cref{eq:optimization_problem_inertia}
can also be transformed into
such an abstract EVI.
Let us note that the existence
of a solution was already proven in
\cite[Theorem 5.1]{groger}
by using essentially the same
transformation into an EVI
as we will do.
However, 
there it was transformed into a
second order EVI
and the maximal monotonicity of the
(slightly different)
operator given therein
was proven in another way.
In contrast,
we consider a first order
EVI
and will provide the
concrete form of the resolvent in
\cref{prop:concrete_form_of_yosida_inertia}
(which will also be used later in
\cref{sec:inertia_oc_oc}
to provide optimality conditions),
the maximal monotonicity of our operator
will then follow easily.
Having transformed the state equation,
we can apply the results from
\cite{paper_abstract}.
We only have to heed two differences
between our EVI and the EVI analyzed in
\cite{paper_abstract}.
First,
our maximal monotone operator
does not fulfill some properties
required in
\cite{paper_abstract},
second,
the given data are more regular in time than in
\cite{paper_abstract},
as we will elaborate on at the end of
\cref{sec:inertia_se_eoas}.
However,
the better regularity in time will compensate
the missing properties of our
maximal monotone operator,
so that the unique existence of a solution
can still be shown
(\cref{thm:existence_of_a_solution_to_the_state_equation_inertia})
and thus we can apply results from
\cite{paper_abstract}.
There is a large list of literature on plasticity with inertia,
we only refer to
\cite{anzellotti1987dynamical,
babadjian2015approximation,
davoli2019dynamic,
dal2014quasistatic,
davoli2019dynamic2,
bartels2008thermoviscoplasticity}
and the references therein.
However,
to best of the author's knowledge,
there exists no contribution
to optimal control of plasticity
with the inertia term,
except in
\cite{walther}. 
We emphasize that this paper is essentially based on
\cite[Part IV]{walther}
and on the transformation idea from
\cite{groger}.

The paper is organized as follows.
After introducing our notation and
standing assumptions in
\cref{sec:2},
we transform the state equation in
\cref{eq:optimization_problem_inertia}
into a first-order EVI,
prove the unique existence
for given data
and
provide regularization and
convergence results in
\cref{sec:inertia_se}.
Afterwards,
in
\cref{sec:inertia_oc},
we analyze the optimal control problem
\cref{eq:optimization_problem_inertia},
show the existence of a global solution,
provide an approximation result via a
regularized problem and finally present
optimality conditions.

\section{Notation and Standing Assumptions}
\label{sec:2}

\paragraph{Notation}
When
$X$
is a normed vector space we
denote its norm by
$\norm{\cdot}{X}$.
For normed vector spaces
$X$
and
$Y$
we denote the space of
linear and continuous functions on
$X$
with values in
$Y$
by
$\LL(X;Y)$.
We abbreviate
$\LL(X) \define \LL(X;X)$.
The dual space of
$X$
is denoted by
$X^* = \LL(X;\R)$.
The inner product
of a Hilber space
$H$
is denoted by
$\scalarproduct{\cdot}{\cdot}{H}$.
For the whole paper, we fix the final time $T > 0$.
For $t > 0$ we denote the Bochner space of square-integrable functions on the
time interval $[0,t]$ by $L^2([0,t]; X)$, the Bochner-Sobolev space by $H^1([0,t]; X)$
and the space of continuous functions by $C([0,t]; X)$.
We furthermore abbreviate
$L^2(X) := L^2([0,T]; X)$,
$H^1(X) := H^1([0,T]; X)$ and
$C(X) := C([0,T]; X)$.
When $G \in \LL(X; Y)$ is a linear and continuous operator, we can define an operator
in $\LL(L^2(X); L^2(Y))$ by $G(u)(t) := G(u(t))$ for all $u \in L^2(X)$ and for almost all $t \in [0,T]$,
we denote this operator also by $G$, that is, $G \in \LL(L^2(X); L^2(Y))$, and analog for
Bochner-Sobolev spaces, i.e., $G \in \LL(H^1(X); H^1(Y))$.
Given a coercive
and symmetric
operator $G \in \LL(H)$ in a real Hilbert space $H$, 
we denote its coercivity constant by $\gamma_G$, i.e., $\scalarproduct{Gh}{h}{H} \geq 
\gamma_G \norm{h}{H}^2$ for all $h \in H$.
With this operator we can define a new scalar product, which induces an equivalent norm, by
$H \times H \ni (h_1, h_2) \mapsto \scalarproduct{Gh_1}{h_2}{H} \in \mathbb{R}$.
We denote the Hilbert space equipped with this scalar product by $H_G$, that is
$\scalarproduct{h_1}{h_2}{H_G} = \scalarproduct{Gh_1}{h_2}{H}$ for all $h_1,h_2 \in H$.
If $p \in [1, \infty]$, then we denote its conjugate exponent by
$p'$, that is $\frac{1}{p} + \frac{1}{p'} = 1$.
Throughout the paper, by $L^p(\Omega; M)$ we denote Lebesgue spaces with values in $M$, 
where $p \in [1, \infty]$ and $M$ is a finite dimensional space.
By
$W^{1,p}(\O;M)$
we denote Sobolev spaces
and
$W^{1,p}_D(\O;M)$
is the subspace containing functions which traces
are zero on $\Gamma_D$.
For the dual spaces of
$W^{1,p'}(\O;M)$
and
$W^{1,p'}_D(\O;M)$
we write
$W^{-1,p}(\O;M)$
and
$W^{-1,p}_D(\O;M)$.
We use the usual abbreviations
$H^1(\O; M) \define W^{1,2}(\O; M)$,
$H^1_D(\O; M) \define W^{1,2}_D(\O; M)$,
$H^{-1}(\O; M) \define W^{-1,2}(\O; M)$
and
$H^{-1}_D(\O; M) \define W^{-1,2}_D(\O; M)$.
Finally, by $\Rdds$, we denote the space of symmetric matrices and 
$c, C>0$ are generic constants.

\subsection*{Standing Assumptions}

The following standing assumptions are tacitly assumed for the rest of the
paper without mentioning them every time.

\paragraph{Domain}
The domain $\Omega\subset\Rd$, $d\in \mathbb{N}$, is bounded 
with Lipschitz boundary $\Gamma$. The boundary consists of two disjoint 
measurable parts $\Gamma_N$ and $\Gamma_D$
such that $\Gamma=\Gamma_N \cup \Gamma_D$. While 
$\Gamma_N$ is a relatively open subset, $\Gamma_D$ is a relatively closed 
subset of $\Gamma$ with positive boundary measure. 
In addition, the set $\Omega \cup \Gamma_N$ is regular in the sense of Gr\"oger, cf.\ \cite{Gro89}.

Furthermore,
the density of $\O$ is given by $\rho > 0$.

\paragraph{Coefficients}
The elasticity tensor and the hardening parameter 
satisfy $\C, \B \in \LL(\Rdds)$
and are symmetric and coercive, i.e.,
there exist constants
$\underline{c}>0$
and
$\underline{b} > 0$ such that
$\scalarproduct{\mathbb{C}\sigma}{\sigma}{\Rdds} \geq \underline{c} \ \norm{\sigma}{\Rdds}^2$
and $\scalarproduct{\mathbb{B}\sigma}{\sigma}{\Rdds} \geq \underline{b} \ \norm{\sigma}{\Rdds}^2$
for all $\sigma \in \Rdds$.

We abbreviate further
\begin{align}
\label{eq:def_D_E_inertia}
	\mathbb{D}
	:=
	\mathbb{B}(\mathbb{C} + \mathbb{B})^{-1} \mathbb{C}
	\in \LL(\Rdds)
	\mediumspace
	\text{and}
	\mediumspace
	\mathbb{E}
	:=
	\mathbb{C}(\mathbb{C} + \mathbb{B})^{-1}
	\in \LL(\Rdds)
\end{align}
and note that
$\mathbb{D}$
is symmetric and coercive,
according to
\cite[Lemma 4.2]{groger}.
Moreover,
for instance,
we denote the adjoint of
$\E$
by
$\E^{\top}$.

\paragraph{Initial data}

We choose
$u_0, v_0 \in H^1_D(\O;\Rd)$
and
$z_0 \in L^2(\O;\Rdds)$
and define
$q_0 \define \mathbb{C} \symnabla u_0 - (\mathbb{C} + \B )z_0
\in L^2(\O;\Rdds)$.
Moreover,
we assume that
$(u_0, v_0, q_0)$
is an element of
$D(\AA)$,
where
$D(\AA)$
is given in
\cref{def:AA_and_spaces_inertia}.
 
\paragraph{Operators}
Throughout the paper, $\symnabla := \frac{1}{2}(\nabla + \nabla^\top): W^{1,p}(\Omega;\Rd) \to L^p(\Omega;\Rdds)$ denotes the linearized strain.
Its restriction to $W^{1,p}_D(\Omega; \Rdds)$ is denoted by the same symbol and, for the adjoint of this restriction, we write
$-\div := (\symnabla)^* : L^{p'}(\Omega;\Rdds) \to W^{-1,p'}_D(\Omega;\Rd)$.

The operator $A \colon L^2(\O;\Rdds) \rightarrow 2^{L^2(\O;\Rdds)}$ is
maximal monotone with domain $D(A)$.
Furthermore,
by
$A_\lambda \colon L^2(\O;\Rdds) \rightarrow L^2(\O;\Rdds)$,
$\lambda > 0$,
we denote the Yosida approximation of
$A$
and by
$R_\lambda = (I + \lambda A)^{-1}$
the resolvent of
$A$,
so that
$A_\lambda = \frac{1}{\lambda}(I - R_\lambda)$.
Moreover,
for every
$\lambda > 0$
the resolvent
$R_\lambda$
can be expressed pointwise,
that is,
there exists
$\tilde{R}_\lambda : \Rdds \to \Rdds$
such that
\begin{align}
\label{eq:resolvent_pointwise_inertia}
	R_\lambda(\tau)(x)
	=
	\tilde{R}_\lambda(\tau(x))
	\mediumspace
	\text{ f.a.a. }
	x \in \O
	\text{ and }
	\forall
	\tau \in L^2(\O;\Rdds).
\end{align}
With a slight abuse of notation we
denote also
$\tilde{R}_\lambda$
by
$R_\lambda$.
It is to be noted that
this is the case
for the subdifferential
of an indicator function
of a pointwise defined set,
where the resolvent is simply
the projection onto this set,
this example will be considered in
\cref{sec:examples}
below.
For further reference on maximal monotone operators, we refer
to~\cite{brezis},~\cite[Ch.~32]{zeidler2a},~\cite[Ch.~55]{zeidler3},
and~\cite[Ch.~55]{showalter}.

\paragraph{Optimization problem}

By
$J : L^2(\HH) \times \XXX_c \rightarrow \R$,
$J(u, v, z, f) \define \Psi(u, v, z) + \frac{\alpha}{2}\norm{f}{\XXX_c}^2$ 
we denote the objective function,
where
$\HH$
is given in
\Cref{def:AA_and_spaces_inertia}
and the control space
$\XXX_c$
is a Hilbert space and embedded into
$H^1(L^2(\O;\Rd))$.
We assume that
$\Psi: L^2(\HH) \to \R$
is weakly lower semicontinuous,
continuous and
bounded from below and
that the Tikhonov parameter
$\alpha$ is a positive constant.

\section{State Equation}
\label{sec:inertia_se}

We begin our investigation
with the state equation.
At first we give the definition of a solution
and then transform the
state equation into an EVI
with
a new (maximal monotone) operator
$\AA$.
In
\cref{sec:inertia_se_eoas}
we prove the existence
of a solution by showing
that the operator
$\AA$ is maximal monotone,
then we can apply
\cite[Theorem 55.A]{zeidler3}.
Finally,
in
\cref{sec:inertia_se_racr}
we can use some results in
\cite{paper_abstract}
to obtain convergence results.

The formal strong formulation
of the state equation
reads
\begin{subequations}\label{eq:state_equation_inertia}
\begin{alignat}{2}
	\rho\bdoubleoverdot{u}
	-
	\nabla \cdot \mathbb{C} (\symnabla u - z)
	&=
	f
	\largespace &&
	\text{ in }
	\O,
	\label{eq:state_equation_a_inertia}
	\\
	\nu \cdot \mathbb{C} (\symnabla u - z) &= 0
	&&
	\text{ on }
	\Gamma_N,
	\label{eq:state_equation_b_inertia}
	\\
	u &= 0
	\largespace &&
	\text{ on }
	\Gamma_D,
	\label{eq:state_equation_c_inertia}
	\\
	\boverdot{z} &\in A(\mathbb{C} \symnabla u - (\C + \B) z)
	\largespace &&
	\text{ in }
	\O,
	\label{eq:state_equation_e_inertia}
	\\
	(u, \boverdot{u}, z)(0) &= (u_0, v_0, z_0)
	\largespace &&
	\text{ in }
	\O.
\end{alignat}
\end{subequations}
Note that we have assumed
in the standing assumptions above
that the density
is constant in $\O$.
It is possible to consider a density which
has a spatial dependency
(that is,
a function from
$\O$
to
$(0, \infty)$),
one has then in particular to verify that
the operator
$Q$,
given in
\cref{def:AA_and_spaces_inertia},
is well defined,
that is,
the multiplication of
$\rho$
(and also
$\nicefrac{1}{\rho}$)
with a Sobolev function
is again a Sobolev function.
However,
for simplicity we assume that
$\rho$
is constant.

We impose the following
assumption for the rest of this
section.

\begin{assumption}[Standing assumption for \cref{sec:inertia_se}]
	Let
	$f \in H^1(L^2(\O;\Rdds))$
	be given.
\end{assumption}

\subsection{Definition and Transformation}
\label{sec:inertia_se_dat}

Let us begin with
the definition of a solution
to the state equation
\cref{eq:state_equation_inertia}.

\begin{definition}[Solution to plasticity with inertia]
\label{def:state_equation_inertia}
	We call
	$u \in H^1(H^1_D(\O;\Rd)) \cap H^2(L^2(\O;\Rd))$
	and
	$z \in H^1(L^2(\O;\Rdds))$
	solution of
	\cref{eq:state_equation_inertia}
	if
	\begin{align}
	\label{eq:state_equation_def_inertia}
	\begin{split}
		\rho\bdoubleoverdot{u}
		-
		\div \mathbb{C} (\symnabla u - z)
		&=
		f,
		\\
		\boverdot{z}
		&\in
		A(\mathbb{C}\symnabla u - (\mathbb{C} + \B )z),
		\\
		(u, \boverdot{u}, z)(0) &= (u_0, v_0, z_0)
	\end{split}
	\end{align}
	holds.
\end{definition}

Before we can transform
the state equation
into an
EVI
we need to reformulate it,
to this end we introduce the following

\begin{definition}[$z$ to $q$ mapping]
\label{def:QQ_and_ZZ_inertia}
	We define
	\begin{align*}
		\QQQ : H^1(\O;\Rd) \times L^2(\O;\Rdds)
		\rightarrow L^2(\O;\Rdds),
		\largespace
		(u,z) \mapsto \mathbb{C}\symnabla u - (\mathbb{C} + \B )z
	\end{align*}
	and its inverse (for fixed
	$u$)
	\begin{align*}
		\ZZZ : H^1(\O;\Rd) \times L^2(\O;\Rdds)
		\rightarrow L^2(\O;\Rdds),
		\largespace
		(u,q) \mapsto (\mathbb{C} + \B )^{-1}(\mathbb{C}\symnabla u - q).
	\end{align*}
\end{definition}

These operators will
also be used to transform the optimal
control problem in
\cref{sec:inertia_oc_eaaooc}
below.

For the following
lemma,
we recall the definition of
$\D$
and
$\E$
given in
the standing assumptions above,
that is,
$\mathbb{D}
=
\mathbb{B}(\mathbb{C} + \mathbb{B})^{-1} \mathbb{C}$
and
$\mathbb{E}
=
\mathbb{C}(\mathbb{C} + \mathbb{B})^{-1}$.

\begin{lemma}[Transformation of $z$ to $q$]
\label{lem:transformation_of_z_to_q_inertia}
	We consider
	\begin{align}
	\label{eq:state_equation_transformed_inertia}
	\begin{split}
		\rho\bdoubleoverdot{u}
		-
		\div(
			\D \symnabla u
			+
			\E q
		)
		&=
		f,
		\\
		(\mathbb{C} + \B )^{-1}\boverdot{q}
		+
		A(q)
		-
		\E^{\top} \symnabla \boverdot{u}
		&\ni
		0,
		\\
		(u, \boverdot{u},q)(0)
		=
		(u_0, v_0, q_0)
		&=
		(u_0, v_0, \QQQ(u_0, z_0))
	\end{split}
	\end{align}
	for functions
	$u \in H^1(H^1_D(\O;\Rd)) \cap H^2(L^2(\O;\Rd))$,
	$q \in H^1(L^2(\O;\Rdds))$.
	Recall that
	$\E^{\top}$
	is the adjoint of
	$\E$.
	Then the following holds:
	
	When
	$(u,z)$
	is a solution of
	\cref{eq:state_equation_inertia},
	then
	$(u, q) = (u, \QQQ(u,z))$
	solves
	\cref{eq:state_equation_transformed_inertia}.
	Vice versa,
	when $(u,q)$
	solves
	\cref{eq:state_equation_transformed_inertia},
	then
	$(u, z) = (u,\ZZZ(u,q))$
	is a solution of
	\cref{eq:state_equation_inertia}.
\end{lemma}
\begin{proof}
	Both implications can be
	immediately obtained by using
	the definition of
	$\QQQ$
	and
	$\ZZZ$
	and inserting
	$z$
	in
	\cref{eq:state_equation_def_inertia}
	and
	$q$
	in
	\cref{eq:state_equation_transformed_inertia},
	respectively
	(note that
	$\mathbb{C} - \mathbb{C}(\mathbb{C} + \B )^{-1}\mathbb{C}
	= (I - \mathbb{C}(\mathbb{C} + \B )^{-1})\mathbb{C}
	= \B (\mathbb{C} + \B )^{-1}\mathbb{C}
	= \D$).
\end{proof}

We are now in the position to introduce the EVI,
respectively the operator
$\AA$.

\begin{definition}[The operator $\AA$]
\label{def:AA_and_spaces_inertia}
	For
	$p \in [1,\infty]$
	we set
	\begin{align*}
		\YY_p
		\define
		W^{1,p}_D(\O;\Rd) \times L^2(\O;\Rd) \times L^p(\O;\Rdds)
		\mediumspace
		\text{and}
		\mediumspace
		\HH \define \YY_2.
	\end{align*}
	The scalar product on
	$\HH$
	is defined by
	\begin{align*}
		&\scalarproduct{(u_1,v_1,q_1)}{(u_2,v_2,q_2)}{\HH}
		\\
		&\qquad
		\define
		\scalarproduct{\D\symnabla u_1}{\symnabla u_2}{L^2(\O;\Rdds)}
		+
		\scalarproduct{v_1}{v_2}{L^2(\O;\Rd)}
		+
		\scalarproduct{q_1}{q_2}{L^2(\O;\Rdds)}
	\end{align*}
	(recall that
	$\D$
	is symmetric and coercive).
	We define
	\begin{align*}
		\AA : D(\AA) \rightarrow 2^\HH,
		\mediumspace
		(u,v,q) \mapsto
		\begin{pmatrix}
			-v
			\\
			-\div(\D\symnabla u + \E q)
			\\
			A(q)
			-
			\E^{\top} \symnabla v
		\end{pmatrix}
	\end{align*}
	with the domain
	\begin{align*}
		D(\AA)
		\define
		\{
			(u,v,q) \in H^1_D(\O;\Rd) \times H^1_D(\O;\Rd) \times D(A)
			:
			\div (\D\symnabla u + \E q)
			\in
			L^2(\O;\Rd)
		\}.
	\end{align*}
	Moreover,
	we set
	\begin{align*}
		R : L^2(\O;\Rd) \rightarrow \YY_\infty,
		\largespace
		f \mapsto (0, f, 0)
	\end{align*}
	and
	\begin{align*}
		Q := (I, (\nicefrac{1}{\rho}) I, \mathbb{C} + \mathbb{B}).
	\end{align*}
\end{definition}

\begin{lemma}[Transformation into an EVI]
\label{lem:transformation_into_an_EVI_inertia}
	The tuple
	$(u,q)$
	solves
	\cref{eq:state_equation_transformed_inertia}
	if and only if
	$(u,v,q) = (u,\boverdot{u},q) \in H^1(\HH)$
	is a solution of
	\begin{align}
	\label{eq:state_equation_as_EVI_inertia}
	\begin{split}
		Q^{-1}(\boverdot{u},\boverdot{v}, \boverdot{q}) + \AA(u,v,q) \ni Rf,
		\mediumspace
		(u, v, q)(0) = (u_0, v_0, q_0).
	\end{split}
	\end{align}
\end{lemma}
\begin{proof}
	This follows immediately from the definition of
	$\AA$.
\end{proof}

\begin{remark}[Consequences of the transformation]
	We note that this transformation
	has some consequences for the optimal control
	problem and its regularization discussed in
	\cref{sec:inertia_oc}.
	A first approach to regularize
	\cref{eq:state_equation_inertia}
	would be to simply regularize the operator
	$A$,
	as we did in the case of elasto plasticity
	in \cite[Part III]{walther}.
	However,
	our approach is different,
	due to the transformation into an EVI
	we can regularize the operator
	$\AA$,
	this is our method in
	\cref{sec:inertia_se_racr}
	and
	\cref{sec:inertia_oc}.
	We also mention that the fact that
	$v = \boverdot{u}$
	will be lost after the regularization
	(cf.
	\cref{cor:concrete_form_of_yosida_inertia}
	and
	\cref{def:AAs_inertia})
	and that we will transform
	our objective function in
	\cref{def:transformed_objective_function},
	so that we obtain an optimal control
	problem with respect to the state
	$(u, v, q)$
	in
	\cref{eq:optimization_problem_inertia_q}.
	The optimality conditions given in
	\cref{thm:KKT_conditions_inertia}
	below are then also formulated for this
	transformed problem.
\end{remark}

\begin{remark}[Neumann surface forces and Dirichlet displacement]
\label{rem:Neumann_and_Dirichlet_inertia}
	Let us shortly discuss some issues with
	possible surface forces and Dirichlet displacements.
	Regarding surface forces,
	they are currently equal to zero and
	contained in the domain
	$D(\AA)$
	by the requirement
	$\div(\D\symnabla u + \E q) \in L^2(\O;\Rd)$.
	Allowing now surface forces which are
	time dependent,
	the domain,
	and thus
	$\AA$
	itself,
	would also depend on the time.
	
	An approach for Dirichlet displacements would be
	to exchange the displacement with a
	``new''
	displacement minus the Dirichlet displacement,
	then one could still define the domain
	$D(\AA)$
	as a subset of
	$H^1_D(\O;\Rd) \times H^1_D(\O;\Rd) \times D(A)$.
	However,
	this would again make the domain
	and the operator itself time dependent
	(the Dirichlet displacement would
	occur also in the operator).
	
	In both cases one could still
	show that the arising operator
	is maximal monotone for a fixed time,
	but for different points in time
	the monotonicity would be perturbed
	by the time dependent functions.
	Having now a closer look at	
	\cref{thm:existence_of_a_solution_to_the_state_equation_inertia}
	below,
	respectively
	\cite[Theorem 55.A]{zeidler3},
	we see that a comparison of two
	different points in time is used
	to derive a priori estimates.
	Following this proof,
	the time depend functions would
	occur and a straightforward
	adaption is not possible.
	
	At this juncture,
	let us also elaborate on the
	underlying spaces of the operator
	$\AA$.
	One might try to exchange
	$L^2(\R;\O)$
	with a negative Sobolev space
	in the definition of
	$\HH$
	to allow surface forces.
	However,
	with this definition of
	$\HH$,
	for instance,
	the proof of
	\cref{lem:AA_to_A_in_scalarproduct_inertia}
	(which is used to show the
	monotonicity of
	$\AA$)
	would not be valid anymore.
	Thus, our choice of
	$\HH$
	seems reasonable.
\end{remark}

\subsection{Existence of a Solution}
\label{sec:inertia_se_eoas}

We prove now the existence of
a solution to
\cref{eq:state_equation_inertia}
by using
an existence result for EVIs involving a
maximal monotone operator given in
\cite[Theorem 55.A]{zeidler3},
thus we need to
show that
$\AA$
is maximal monotone.
Since the monotonicity of
$\AA$
can be easily obtained
(cf.
\cref{lem:AA_to_A_in_scalarproduct_inertia}),
it remains to prove that the resolvent exists
(cf. the proof of
\cref{prop:AA_maximal_monotone_inertia}).
For this
it is sufficient to show the existence
of a solution to
\cref{eq:TR_def_inertia}
in the case
$p = 2$.
However,
since the existence and
Lipschitz continuity for
$p > 2$ is needed to derive
optimality conditions in
\cref{sec:inertia_oc_oc},
we already provide the
following corollary
for later needed results.

\begin{corollary}[Extended nonlinear elasticity]
\label{cor:W1r_existence_inertia}
	Let
	$\lambda > 0$
	and
	$p \in [2,\overline{p}]$,
	where
	$\overline{p}$
	is from
	\cite[Theorem 1.1]{herzog},
	with
	$2 - \frac{d}{2} \geq -\frac{d}{\overline{p}}$.
	We assume that there exist
	$m,M,D \in \R$,
	$D \geq 0 < m \leq M$,
	such that the family of functions
	$\{ b_\sigma : \O \times \Rdds
	\rightarrow \Rdds \}_{\sigma \in \Rdds}$
	has the following properties:
	\begin{align}
		&
		b_{0}(\cdot,0) \in L^\infty(\O; \Rdds),
		\label{eq:W1r_existence_inertia_1}
		\\
		&
		b_{\sigma}(\cdot,\tau) \text{ is measurable},
		\label{eq:W1r_existence_inertia_2}
		\\
		&
		(b_\sigma(x,\tau)
		-
		b_{\overline{\sigma}}(x,\overline{\tau}))
		\colon
		(\tau - \overline{\tau})
		+
		D (| \sigma - \overline{\sigma} |
		+
		| \tau - \overline{\tau} |)
		| \sigma - \overline{\sigma} |
		\geq
		m | \tau - \overline{\tau} |^2,
		\label{eq:W1r_existence_inertia_3}
		\\
		&
		| b_\sigma(x,\tau)
		-
		b_{\overline{\sigma}}(x,\overline{\tau}) |
		\leq
		M \Big{(}
			| \tau - \overline{\tau} |
			+
			| \sigma - \overline{\sigma} |
		\Big{)}
		\label{eq:W1r_existence_inertia_4}
	\end{align}
	for almost all
	$x \in \O$
	and all
	$\sigma, \overline{\sigma}, \tau, \overline{\tau}
	\in \Rdds$.
	
	Then for every
	$\varphi \in L^p(\O;\Rdds)$
	and
	$L \in W^{-1,p}_D(\O;\Rd)$
	there exists a unique solution
	$u \in W^{1,p}_D(\O;\Rd)$
	of
	\begin{align*}
		-\div b_{\varphi}(\cdot,\symnabla u)
		+
		\frac{u}{\lambda^2}
		=
		L.
	\end{align*}
	Moreover,
	there exists a constant
	$C$
	such that the inequality
	\begin{align*}
		\norm{u_1 - u_2}{W^{1,p}(\O;\Rd)}
		\leq
		C \Big{(}
			\norm{\varphi_1 - \varphi_2}{L^p(\O;\Rdds)}
			+
			\norm{L_1 - L_2}{W^{-1,p}_D(\O;\Rd)}
		\Big{)}
	\end{align*}
	holds for all
	$\varphi_1, \varphi_2 \in L^p(\O;\Rdds)$
	and
	$L_1, L_2 \in W^{-1,p}_D(\O;\Rd)$,
	where
	$u_1$
	and
	$u_2$
	are the solutions with respect to
	$(\varphi_1, L_1)$
	and
	$(\varphi_2, L_2)$.
\end{corollary}
\begin{proof}
	Note that
	$b_\sigma(\cdot, \tau) \in L^p(\O;\Rdds)$
	holds for all
	$\tau, \sigma \in L^p(\O;\Rdds)$
	(and in fact for all
	$p \in [1, \infty]$),
	which follows from
	\cref{eq:W1r_existence_inertia_1},
	\cref{eq:W1r_existence_inertia_2}
	and
	\cref{eq:W1r_existence_inertia_4}
	(taking into account that a pointwise
	limit of measurable functions
	is also measurable,
	see
	\cite[Corollary 3.1.5]{wachsmuth}).
	
	Let us at first consider the case
	$p = 2$.
	Then the existence of a solution follows from
	the Browder-Minty theorem,
	Korn's inequality
	and
	the Poincar\'{e} inequality.
	In order to verify the inequality,
	let
	$\varphi_1, \varphi_2
	\in L^2(\O;\Rdds)$,
	$L_1, L_2 \in H^{-1}_D(\O;\Rd)$
	and
	$u_1, u_2 \in H^1_D(\O;\Rd)$
	the corresponding solutions.
	Then we obtain
	\begin{align*}
		\dualpair{L_1 - L_2}{u_1 - u_2}{}		
		&=
		\scalarproduct{b_{\varphi_1(\cdot)}(\cdot,\symnabla u_1(\cdot))
		- b_{\varphi_2(\cdot)}(\cdot,\symnabla u_2(\cdot))}
		{\symnabla (u_1 - u_2)}{L^2(\O;\Rdds)}
		\\
		&\qquad
		+
		\bignorm{\frac{u_1 - u_2}{\lambda}}{L^2(\O;\Rd)}^2
		\\
		&\geq
		m
		\norm{\symnabla (u_1 - u_2)}{L^2(\O;\Rdds)}^2
		-
		D \norm{\varphi_1 - \varphi_2}{L^2(\O;\Rdds)}^2
		\\
		&\largespace
		-
		D \int_\O
			|\symnabla (u_1 - u_2)| \
			|\varphi_1 - \varphi_2|
		+
		\frac{1}{\lambda^2}
		\norm{u_1 - u_2}{L^2(\O;\Rd)}^2.
	\end{align*}
	Using now
	\begin{align*}
		\dualpair{L_1 - L_2}{u_1 - u_2}{}
		\leq
		\norm{L_1 - L_2}{W^{-1,2}_D(\O;\Rd)}
		\norm{u_1 - u_2}{W^{1,2}(\O;\Rd)}
	\end{align*}
	and
	\begin{align*}
		D \int_\O
			|\symnabla (u_1 - u_2)| \
			|\varphi_1 - \varphi_2|
		\leq
		D
		\norm{\symnabla (u_1 - u_2)}{L^2(\O;\Rd)}
		\norm{\varphi_1 - \varphi_2}{L^2(\O;\Rd)}
	\end{align*}
	and Young's inequality,
	yields
	\begin{align*}
		\norm{u_1 - u_2}{W^{1,2}(\O;\Rd)}^2
		&\leq
		\overline{C} \Big{(}
			\norm{\varphi_1 - \varphi_2}{L^2(\O;\Rdds)}^2
			+
			\norm{L_1 - L_2}{W^{-1,2}_D(\O;\Rd)}^2
		\Big{)}
		\\
		&\leq
		\overline{C} \Big{(}
			\norm{\varphi_1 - \varphi_2}{L^2(\O;\Rdds)}
			+
			\norm{L_1 - L_2}{W^{-1,2}_D(\O;\Rd)}
		\Big{)}^2
	\end{align*}
	for a certain positive constant
	$\overline{C}$,
	hence,
	the asserted inequality is fulfilled.
	
	For the general case let now
	$\varphi \in L^\infty(\O; \Rdds)$
	and
	$L \in W^{-1,p}_D(\O;\Rd)$,
	we define
	$b : \O \times \Rdds \rightarrow \Rdds$
	by
	\begin{align*}
		b(x,\tau) \define b_{\varphi(x)}(x,\tau)
	\end{align*}
	and
	$L_u \in W^{-1,p}_D(\O;\Rd)$
	by
	\begin{align*}
		\dualpair{L_u}{v}{}
		\define
		\dualpair{L}{v}{}
		-
		\frac{1}{\lambda^2}
		\scalarproduct{u}{v}{L^2(\O;\Rd)},
	\end{align*}
	where
	$u \in H^1_D(\O;\Rd)
	\embed
	L^q(\O;\Rd)$,
	with
	$1 - \frac{d}{2} = -\frac{d}{q}$
	when
	$d > 2$,
	$q = 2$
	when
	$d = 2$
	and
	$q = \infty$
	when
	$d = 1$,
	is the solution in the case
	$p = 2$
	and
	$v \in W^{1,p'}(\O;\Rd)
	\embed
	L^{q'}(\O;\Rd)$
	(note that
	$1 - \frac{d}{p'} + \frac{d}{q'}
	= 1 + \frac{d}{p} - \frac{d}{q}
	= 2 + \frac{d}{p} - \frac{d}{2}
	\geq 0$
	when
	$d > 2$
	and
	$1 - \frac{d}{p'} + \frac{d}{q'}
	>
	0$
	otherwise).
	We can now apply
	\cite[Theorem 1.1]{herzog}
	(here we need
	$\varphi \in L^\infty(\O; \Rdds)$
	to satisfy
	\cite[(1.6a)]{herzog},
	the other requirements in
	\cite[Assumption 1.5]{herzog}
	are obviously fulfilled due to
	\crefrange{eq:W1r_existence_inertia_1}
	{eq:W1r_existence_inertia_4})
	to obtain
	$\overline{u} \in W^{1,p}_D(\O;\Rd)$
	such that
	\begin{align*}
		\scalarproduct{b(\cdot,\symnabla \overline{u})}
		{\symnabla v}{L^2(\O;\Rdds)}
		=
		\dualpair{L_u}{v}{},
	\end{align*}
	that is,
	\begin{align*}
		\scalarproduct{b_{\varphi}(\cdot,\symnabla \overline{u})}
		{\symnabla v}{L^2(\O;\Rdds)}
		+
		\frac{1}{\lambda^2}
		\scalarproduct{u}{v}{L^2(\O;\Rd)}
		=
		\dualpair{L}{v}{},
	\end{align*}
	holds for all
	$v \in W^{1,p'}_D(\O;\Rd)$,
	we get in particular
	$u = \overline{u} \in W^{1,p}_D(\O;\Rd)$
	since
	$u$
	is the unique solution of the equation above
	for all
	$v \in H^1_D(\O;\Rd)$.
	
	To prove the asserted inequality let
	$\varphi_1,\varphi_2 \in L^\infty(\O; \Rdds)$,
	$L_1,L_2 \in W^{-1,p}_D(\O;\Rd)$
	and
	$u_1,u_2 \in W^{1,p}(\O;\Rd)$
	the corresponding solutions and define
	$L_{u_1}, L_{u_2}$
	as before.
	Having a closer look at the proof of
	\cite[Theorem 1.1]{herzog},
	respectively
	\cite[Theorem 1]{Gro89},
	one can see that there exists
	a constant
	$c > 0$,
	depending only on
	$p, m$
	and
	$M$
	(thus not on
	$L_1,L_2,\varphi_1,\varphi_2$),
	such that
	\begin{align*}
		\norm{u_1 - u_2}{W^{1,p}(\O;\Rd)}
		\leq
		c \norm{A_1(u_2) - A_2(u_2) - L_{u_1} + L_{u_2}}{W^{-1,p}_D(\O;\Rd)},
	\end{align*}
	where
	$A_i : W^{1,p}(\O;\Rd)
	\rightarrow W^{-1,p}_D(\O;\Rd)$
	is defined by
	\begin{align*}
		\dualpair{A_i(v_1)}{v_2}{}
		\define 
		\scalarproduct{b_{\varphi_i}(\cdot, \symnabla v_1)}
		{\symnabla v_2}{L^2(\O;\Rdds)}
	\end{align*}
	for all
	$v_1 \in W^{1,p}(\O;\Rd)$,
	$v_2 \in W^{1,p'}(\O;\Rd)$
	and for
	$i \in \{ 1,2 \}$.
	We finally obtain
	\begin{align*}
		\norm{u_1 - u_2}{W^{1,p}(\O;\Rd)}
		&\leq
		c \Big{(}
			\norm{A_1(u_2) - A_2(u_2)}{W^{-1,p}_D(\O;\Rd)}
			+ 
			\norm{L_{u_1} - L_{u_2}}{W^{-1,p}_D(\O;\Rd)}
		\Big{)}
		\\
		&\leq
		c \Big{(}
			M \norm{\varphi_1 - \varphi_2}{L^p(\O;\Rdds)}
			+ 
			\norm{L_1 - L_2}{W^{-1,p}_D(\O;\Rd)}
			\\
			&\largespace\largespace
			+
			\frac{C}{\lambda^2}
			\norm{u_1 - u_2}{H^1(\O;\Rd)}
		\Big{)},
	\end{align*}
	where we have used again the embeddings
	$H^1(\O;\Rd)
	\embed
	L^q(\O;\Rd)$
	and
	$W^{1,p'}(\O;\Rd)
	\embed
	L^{q'}(\O;\Rd)$.
	Taking into account that
	the assertion is already proven in the case
	$p = 2$,
	we see that the desired inequality holds.
	
	One can now obtain the result for all
	$\varphi_1,\varphi_2 \in L^p(\O;\Rdds)$
	by an approximation
	(using the just proven inequality to see
	that the corresponding sequence
	$u_n$
	is a Cauchy sequence).
\end{proof}

The operator
$R_0$
in the following
proposition
will later be
the resolvent,
or a smoothed version of
the resolvent,
of
$A$
and should not be
confused with
$R$
from
\cref{def:AA_and_spaces_inertia}.

\begin{proposition}[Solution operator $\TT_{R_0}$]
\label{prop:TR_lipschitz_inertia}
	Let
	$\lambda > 0$
	and
	$p \geq 2$
	as in
	\cref{cor:W1r_existence_inertia}
	and
	$h = (h_1,h_2,h_3) \in \YY_p$.
	Moreover,
	let
	$R_0 : \Rdds \rightarrow \Rdds$
	be Lipschitz continuous and monotone.
	Then there exists a unique solution
	$u \in W^{1,p}_D(\O;\Rd)$
	of
	\begin{align}
	\label{eq:TR_def_inertia}
	\begin{split}
		-\div (
			\D \symnabla u
			+ \E R_0(\E^{\top} \symnabla (u - h_1) + h_3)
		)
		=
		\frac{h_2}{\lambda} + \frac{h_1 - u}{\lambda^2}.
	\end{split}
	\end{align}
	We denote the solution
	operator of this equation by
	$\TT_{R_0} : \YY_p \rightarrow W^{1,p}_D(\O;\Rd)$,
	that is,
	$\TT_{R_0}(h) = u$.
	Furthermore,
	$\TT_{R_0}$
	is Lipschitz continuous.
	Note that the dependency of
	$\TT_{R_0}$
	on
	$\lambda$
	and
	$p$
	will always be clear from the context.
\end{proposition}
\begin{proof}
	For all
	$\sigma \in \Rdds$
	we define
	$b_\sigma : \O \times \Rdds \rightarrow \Rdds$
	by
	\begin{align*}
		b_\sigma (x,\tau)
		\define
		\D \tau
		+
		\E R_0(\E^{\top}\tau + \sigma),
	\end{align*}
	then the assertion follows from
	\cref{cor:W1r_existence_inertia}
	(with
	$\varphi \define -\E^{\top} \symnabla h_1 + h_3$
	for a given
	$h \in \YY_p$),
	let us only prove that
	\cref{eq:W1r_existence_inertia_3}
	is fulfilled,
	the other requirements can be easily checked.
	To this end let
	$\sigma, \overline{\sigma}, \tau, \overline{\tau} \in \Rdds$,
	then
	\begin{align*}
		(b_\sigma(x,\tau)
		&
		-
		b_{\overline{\sigma}}(x,\overline{\tau}))
		\colon
		(\tau - \overline{\tau})
		\\
		&\geq
		\gamma_{\D}
		| \tau - \overline{\tau} |^2
		+
		\Big{(}
			R_0(\E^{\top} \tau + \sigma)
			-
			R_0(\E^{\top} \overline{\tau}
			+
			\overline{\sigma})
		\Big{)}
		\colon
		\Big{(}
			\E^{\top} (\tau - \overline{\tau})
			+
			(\sigma - \overline{\sigma})
		\Big{)}
		\\
		&\largespace
		-
		\Big{(}
			R_0(\E^{\top} \tau + \sigma)
			-
			R_0(\E^{\top} \overline{\tau}
			+
			\overline{\sigma})
		\Big{)}
		\colon
		(\sigma - \overline{\sigma})
		\\
		&\geq
		\gamma_{\D}
		| \tau - \overline{\tau} |^2
		-
		L_{R_0} | \sigma - \overline{\sigma} |^2
		-
		L_{R_0} \norm{\E}{} \
		| \tau - \overline{\tau} | \
		| \sigma - \overline{\sigma} |
	\end{align*}
	holds,
	where
	$L_{R_0}$
	is the Lipschitz constant of
	$R_0$.
\end{proof}

Note that
$R_\lambda : \Rdds \to \Rdds$
fulfills the requirements in
\cref{prop:TR_lipschitz_inertia}
since
$R_\lambda : L^2(\O;\Rdds) \to L^2(\O;\Rdds)$
is Lipschitz continuous
and also monotone
(cf.
\cite[Proposition 55.1 (ii)
and
Proposition 55.2 (a)]{zeidler3})
and due to
\cref{eq:resolvent_pointwise_inertia}
these properties carry over to
$R_\lambda : \Rdds \to \Rdds$.

Let us also mention that
$R_0$
in
\cref{prop:TR_lipschitz_inertia} 
does not have to be monotone,
the inequality
\begin{align*}
	(R_0(a) - R_0(b)) \colon (a - b) \geq -\varepsilon |a - b|^2
\end{align*}
for
$a,b \in \Rdds$
with
$\varepsilon < \nicefrac{\gamma_{\D}}{\norm{\E^{\top}}{}^2}$
would be sufficient.

We can now prove the existence
of the resolvent of
$\AA$,
from which we can then derive
the maximal monotonicity of
$\AA$
in
\cref{prop:AA_maximal_monotone_inertia}
below.

\begin{proposition}[Existence of the resolvent of $\AA$]
\label{prop:concrete_form_of_yosida_inertia}
	For every
	$\lambda > 0$
	and
	$h = 
	(h_1,h_2,h_3) \in \HH$,
	the tuple
	\begin{align*}
		\begin{pmatrix}
			u \\ v \\ q
		\end{pmatrix}
		=
		\begin{pmatrix}
			\TT_{R_\lambda}(h)
			\\
			\frac{1}{\lambda}(\TT_{R_\lambda}(h) - h_1)
			\\
			R_\lambda(\E^{\top}\symnabla(\TT_{R_\lambda}(h) - h_1) + h_3)
		\end{pmatrix}
	\end{align*} 
	is contained in
	$D(\AA)$
	and the unique solution of
	$(u,v,q) + \lambda \AA(u,v,q) \ni h$.
\end{proposition}
\begin{proof}
	Using the definition of
	$\TT_{R_\lambda}$
	we get
	\begin{align*}
		-\lambda \div(\D \symnabla u + \E q)
		=
		h_2 - v,
	\end{align*}
	which is the second row in
	$(u,v,q) + \lambda \AA(u,v,q) \ni h$
	and we also get
	$(u,v,q) \in D(\AA)$
	(note that
	$rg(R_\lambda) \subset D(A)$).
	That the first and last row in
	$(u,v,q) + \lambda \AA(u,v,q) \ni h$
	is also fulfilled follows immediately from
	the definitions of
	$u$,
	$v$
	and
	$q$.
	
	Furthermore,
	when
	$(u, v, q)$
	is a solution of
	$(u,v,q) + \lambda \AA(u,v,q) \ni h$,
	then one verifies analog that
	$(u,v,q)$
	must have the claimed form,
	therefore the uniqueness follows
	from the uniqueness of a solution to
	\cref{eq:TR_def_inertia}.
\end{proof}

\begin{lemma}[Monotonicity of $\AA$]
\label{lem:AA_to_A_in_scalarproduct_inertia}
	The equation
	\begin{align*}
		&\scalarproduct
		{\AA(u_1,v_1,q_1) - \AA(u_2,v_2,q_2)}
		{(u_1,v_1,q_1) - (u_2,v_2,q_2)}{\HH}
		\\
		&\qquad=
		\scalarproduct{A(q_1) - A(q_2)}{q_1 - q_2}{L^2(\O;\Rdds)}
	\end{align*}
	holds for all
	$(u_1,v_1,q_1), (u_2,v_2,q_2) \in D(\AA)$.
\end{lemma}
\begin{proof}
	Using the definition of
	$\AA$
	and the scalar product in
	$\HH$
	we obtain
	\begin{align*}
		&\scalarproduct
		{\AA(u_1,v_1,q_1)}
		{(u_1,v_1,q_1) - (u_2,v_2,q_2)}{\HH}
		\\
		&\largespace
		=
		-\scalarproduct{\D\symnabla v_1}{\symnabla (u_1 - u_2)}{L^2(\O;\Rdds)}
		-
		\scalarproduct{\div(\D\symnabla u_1 + \E q_1)}
		{v_1 - v_2}{L^2(\O;\Rd)} \\
		&\largespace\largespace\largespace
		+
		\scalarproduct
		{A(q_1) - \E^{\top}}{\symnabla v_1}
		{q_1 - q_2}
		{L^2(\O;\Rdds)}
		\\
		&\largespace
		=
		\scalarproduct{\D\symnabla v_1}{\symnabla u_2}{L^2(\O;\Rdds)}
		-
		\scalarproduct{\D\symnabla u_1}{\symnabla v_2}{L^2(\O;\Rdds)}
		-
		\scalarproduct{\E^{\top}\symnabla v_2}{q_1}{L^2(\O;\Rdds)}
		\\
		&\largespace\largespace\largespace
		+
		\scalarproduct{\E^{\top}\symnabla v_1}{q_2}{L^2(\O;\Rdds)}
		+
		\scalarproduct{A(q_1)}{q_1 - q_2}{L^2(\O;\Rdds)},
	\end{align*}
	evaluating now
	$\scalarproduct
	{\AA(u_2,v_2,q_2)}
	{(u_1,v_1,q_1) - (u_2,v_2,q_2)}{\HH}$
	and taking the difference yields
	the assertion.
\end{proof}

\begin{proposition}[$\AA$ is maximal monotone]
\label{prop:AA_maximal_monotone_inertia}
	The operator
	$\AA : \HH \rightarrow 2^\HH$
	is maximal monotone.
\end{proposition}
\begin{proof}
	The monotonicity of
	$\AA$
	follows immediately from
	\cref{lem:AA_to_A_in_scalarproduct_inertia}
	and the monotonicity of
	$A$.
	
	To prove that
	$\AA$
	is maximal monotone,
	it is, 
	according to
	\cite[Proposition 55.1 (B)]{zeidler3},
	sufficient that
	$R(I + \AA) = \HH$,
	that is,
	we have to show that for every
	$(h_1,h_2,h_3) \in \HH$
	there exists
	$(u,v,q) \in D(\AA)$
	such that
	$(u,v,q) + \AA(u,v,q) \ni (h_1,h_2,h_3)$.
	This follows from
	\cref{prop:concrete_form_of_yosida_inertia}
	with
	$\lambda = 1$.
\end{proof}

In what follows it is convenient to
give the integration operator a name.

\begin{definition}[Integration operator]
\label{def:integration_operator_inertia}
	We define
	$\mathcal{F} : H^1(L^2(\O;\Rd))
	\rightarrow	
	H^2(L^2(\O;\Rd))$
	by
	$(\FF f)(t)
	\define \int_0^t f(s) ds$
	for all
	$f \in H^1(L^2(\O;\Rd))$.
	Moreover,
	we abbreviate
	$\FF_\rho \define \nicefrac{\FF}{\rho}$.
	As usual,
	we denote the operators
	with different inverse images
	and ranges with the same symbol,
	for instance
	$\mathcal{F} : L^2(L^2(\O;\Rd))
	\rightarrow H^1(L^2(\O;\Rd))$.
\end{definition}

\begin{theorem}[Existence of a solution to the state equation]
\label{thm:existence_of_a_solution_to_the_state_equation_inertia}
	There exists a unique solution
	$(u,v,q) \in H^1(\HH)$
	of
	\cref{eq:state_equation_as_EVI_inertia}.
	Moreover,
	the inequality
	\begin{align*}
		\norm{(\boverdot{u}, \boverdot{v}, \boverdot{q})}{L^2(\HH)}
		\leq C (1 + \norm{f}{H^1(L^2(\O;\Rd))})
	\end{align*}
	holds, where the constant $C$ does not depend on $f$.
\end{theorem}
\begin{proof}
	The tuple
	$(u,v,q)$
	is a solution of
	\cref{eq:state_equation_as_EVI_inertia}
	if and only if
	$w := (u,v,q)$
	solves
	\begin{align*}
		\boverdot{w}
		+
		\tilde{\AA}(w)
		\ni
		\frac{1}{\rho}Rf,
		\mediumspace
		w(0) = (u_0,v_0,q_0),
	\end{align*}
	with
	$\tilde{\AA} := Q\AA$.
	One easily verifies that
	$\tilde{\AA}$
	is a maximal monotone operator with respect to
	$\HH_{Q^{-1}}$
	(that is,
	the space
	$\HH$
	equipped with the scalar product
	$\scalarproduct{Q^{-1}\cdot}{\cdot}{\HH_{Q^{-1}}}$),
	thus we can apply
	\cite[Theorem 55.A]{zeidler3}
	to obtain a solution
	$w \in H^1(\HH)$.
	Moreover,
	as can be seen in the proof of
	\cite[Theorem 55.A]{zeidler3}
	we get
	\begin{align*}
		\sqrt{\gamma_{Q^{-1}}}
		\norm{\boverdot{w}_\lambda}{C(\HH)}
		\leq
		\norm{\boverdot{w}_\lambda}{C(\HH_{Q^{-1}})}
		\leq
		C (1 + \frac{1}{\rho}\norm{Rf}{H^1(\HH)}))
		=
		C (1 + \frac{1}{\rho}\norm{f}{H^1(L^2(\Omega;\Rd))})
	\end{align*}
	for all
	$\lambda > 0$,
	where
	$w_\lambda$
	is the solution of
	\begin{align*}
		\boverdot{w}_\lambda
		+
		\tilde{\AA}_\lambda(w_\lambda)
		=
		\frac{1}{\rho}Rf,
		\mediumspace w_\lambda(0)
		=
		(u_0,v_0,q_0).
	\end{align*}
	Since
	$w_\lambda \rightharpoonup w$
	in
	$H^1(\HH)$,
	we obtain the desired inequality.
\end{proof}

\begin{remark}[$\AA$ is not a subdifferential]
\label{rem:AA_is_not_a_subdifferential_inertia}
	Let us show that
	the maximal monotone operator
	$\AA : \HH \rightarrow \HH$
	is not a subdifferential,
	that is,
	there exists no proper,
	convex and lower semicontinuous function
	$\Phi : \HH \rightarrow (-\infty, \infty]$
	such that
	\begin{align*}
		\AA(u,v,q)
		&=
		\partial\Phi(u,v,q)
		=
		\{
			(c, d, e) \in \HH
			:
			\Phi(\hat{u}, \hat{v}, \hat{q})
			\geq
			\Phi(u,v,q)
			\\
			&\largespace\largespace\largespace\largespace
			+
			\scalarproduct{(c, d, e)}{(\hat{u} - u, \hat{v} - v, \hat{q} - q)}{\HH} \
			\forall (\hat{u}, \hat{v}, \hat{q})
			\in \HH
		\}
	\end{align*}
	holds for all
	$(u,v,q) \in \HH$.
	In fact,
	there exists even not any function
	$\Phi : \HH \rightarrow (-\infty, \infty]$
	such that the equation above holds,
	which can be seen as follows:
	
	Let us assume that such a
	$\Phi$
	exists and recall that
	$(u_0,v_0,q_0) \in D(\AA)$.
	Then,
	using
	\cref{lem:AA_to_A_in_scalarproduct_inertia}
	with
	$(u_1,v_1,q_1) = (u + u_0,v,q_0)$
	and
	$(u_2,v_2,q_2) = (u_0,0,q_0)$,
	\begin{align*}
		\Phi(u_0,0,q_0)
		&\geq
		\Phi(u + u_0,v,q_0)
		-
		\scalarproduct{\AA(u + u_0,v,q_0)}{(u,v,0)}{\HH}
		\\
		&=
		\Phi(u + u_0,v,q_0)
		-
		\scalarproduct{\AA(u_0,0,q_0)}{(u,v,0)}{\HH}
		\\
		&\geq
		\Phi(u_0,0,q_0)
		+
		\scalarproduct{\AA(u_0,0,q_0)}{(u,v,0)}{\HH}
		-
		\scalarproduct{\AA(u_0,0,q_0)}{(u,v,0)}{\HH}
		\\
		&=
		\Phi(u_0,0,q_0)
	\end{align*}
	holds for all
	$(u,v)$
	such that
	$(u + u_0,v,q_0) \in D(\AA)$,
	hence,
	\begin{align*}
		\scalarproduct
		{\AA(u_0, 0, q_0)}
		{(\hat{u} - u, \hat{v} - v, 0)}
		{\HH}
		&=
		\Phi(\hat{u} + u_0, \hat{v}, q_0)
		-
		\Phi(u + u_0, v, q_0)
		\\
		&\geq
		\scalarproduct
		{\AA(u + u_0,v,q_0)}
		{(\hat{u} - u,\hat{v} - v,0)}
		{\HH}
	\end{align*}
	which gives
	\begin{align*}
		0
		\geq
		\scalarproduct
		{\D\symnabla u}
		{\symnabla \hat{v}}
		{\HH}
		-
		\scalarproduct
		{\D\symnabla \hat{u}}
		{\symnabla v}
		{\HH}
	\end{align*}
	for all
	$(\hat{u},\hat{v}), (u,v)$
	such that
	$(u + u_0,v,q_0),
	(\hat{u} + u_0,\hat{v},q_0)
	\in D(\AA)$.
	Choosing now an arbitrary
	$u \in C_c^\infty(\O;\Rd)$,
	$u \neq 0$,
	$\hat{v} = u$
	and
	$\hat{u} = v = 0$,
	we obtain the desired contradiction.
\end{remark}

In light of
\cref{rem:AA_is_not_a_subdifferential_inertia},
the case of plasticity with inertia essentially
differs from the EVI analyzed in
\cite{paper_abstract}
in two aspects.
First,
we have more regularity in time
as explained after
\cref{lem:transformation_into_an_EVI_inertia}.
Second,
we lose a certain boundedness of the
maximal monotone operator,
which was assumed in
\cite[Sect. 2]{paper_abstract},
and it is not a subdifferential.

It is also to be noted that
\cref{rem:AA_is_not_a_subdifferential_inertia}
is independent of the operator
$A$.

\subsection{Regularization and Convergence Results}
\label{sec:inertia_se_racr}

As already pointed out earlier,
\cref{eq:state_equation_as_EVI_inertia}
can be transformed into the EVI which was
analyzed in
\cite{paper_abstract},
namely
\begin{align}
\label{eq:state_equation_as_EVI_inertiap}
	\boverdot{p} \in \AA(R\FF_\rho f - Qp), \mediumspace p(0) = p_0,
\end{align}
where we set $p_0 := -Q^{-1}(u_0, v_0, q_0)$.
We can observe that when $(u,v,q)$ is
a solution of \cref{eq:state_equation_as_EVI_inertia}, then $p := R\FF f - Q^{-1}(u,v,q)$ is a solution of
\cref{eq:state_equation_as_EVI_inertiap} and when $p$ is a solution of \cref{eq:state_equation_as_EVI_inertiap},
then $(u,v,q) := R \FF_\rho f - Qp$ is a solution of \cref{eq:state_equation_as_EVI_inertia}.

Thanks to this transformation we can use several results from \cite{paper_abstract}
(namely Lemma 3.7, Lemma 3.8 and Proposition 3.5), with which
the derivation of convergence results will be
an easy task. Let us emphasize that the maximal monotone operator therein was assumed to
have a closed domain and that
\begin{equation*}
	A^0 : D(A) \rightarrow \HH, \quad 
	h \mapsto \argmin_{v \in A(h)} \norm{v}{\HH}
\end{equation*}
is bounded on bounded sets. Clearly, both assumptions are not fulfilled for $\AA$, however, they were only
needed to prove \cite[Theorem 3.3]{paper_abstract}, which can be replaced by \cref{thm:existence_of_a_solution_to_the_state_equation_inertia}, cf. also
\cite[Remark 3.13]{paper_abstract}. Therefore we can still apply the above mentioned results.

\begin{theorem}[Weak convergence of the state]
\label{thm:weak_convergence_of_the_state_inertia}
	Let $f \in H^1(L^2(\O;\Rd))$,
	$\sequence{f}{n} \subset H^1(L^2(\O;\Rd))$
	such that $f_n \rightharpoonup f$ in $H^1(L^2(\O;\Rd))$ and $\FF f_n \rightarrow \FF F$
	in $L^1(L^2(\O;\Rd))$. Moreover, let $(u,v,q) \in H^1(\HH)$
	be the solution of \eqref{eq:state_equation_as_EVI_inertia} and $(u_n,v_n,q_n) \in H^1(\HH)$,
	for every $n \in \mathbb{N}$, either the solution of
	\begin{align*}
		Q^{-1}(\boverdot{u}_n,\boverdot{v}_n, \boverdot{q}_n) + \AA(u_n,v_n,q_n) \ni Rf_n,
		\mediumspace (u_n,v_n,q_n)(0) = (u_0,v_0,q_0)
	\end{align*}
	or
	\begin{align*}
		Q^{-1}(\boverdot{u}_n,\boverdot{v}_n, \boverdot{q}_n) + \AA_{\lambda_n}(u_n,v_n,q_n) = Rf_n,
		\mediumspace (u_n,v_n,q_n)(0) = (u_0,v_0,q_0),
	\end{align*}
	where $\sequence{\lambda}{n} \subset (0,\infty)$, $\lambda_n \searrow 0$.
	
	Then $(u_n,v_n,q_n) \rightharpoonup (u,v,q)$ in $H^1(\HH)$ and
	$(u_n,v_n,q_n) \rightarrow (u,v,q)$ in $C(H^1(\Omega;\Rd)) \times L^1(L^2(\O;\Rd)) \times C(L^2(\O;\Rdds))$.
	If additionally $\FF f_n \rightarrow \FF f$ in $C(L^2(\O;\Rdds))$, then
	$v_n \rightarrow v$ in $C(L^2(\O;\Rd))$.
\end{theorem}
\begin{proof}
	The function
	$p := R\FF f - Q^{-1}(u,v,q) \in H^1(\HH)$ is the unique solution
	of \eqref{eq:state_equation_as_EVI_inertiap} and
	$p_n := R\FF f_n - Q^{-1}(u_n,v_n,q_n) \in H^1(\HH)$
	either the unique solution of
	\begin{align*}
		\boverdot{p}_n \in \AA(R\FF_\rho f_n - Qp_n), \mediumspace p_n(0) = p_0.
	\end{align*}
	or
	\begin{align*}
		\boverdot{p}_n = \AA_{\lambda_n}(R\FF_\rho f_n - Qp_n), \mediumspace p_n(0) = p_0.
	\end{align*}
	Thanks to \cref{thm:existence_of_a_solution_to_the_state_equation_inertia}, $(\boverdot{u}_n, \boverdot{v}_n, \boverdot{q}_n)$
	is bounded in $L^2(\HH)$.
	We can now apply \cite[Lemma 3.7]{paper_abstract} with $A_n = \AA$
	and $A_n = \AA_{\lambda_n}$ (note that we can choose $A_n = \AA_{\lambda_n}$
	according to \cite[Lemma 3.8]{paper_abstract}) to obtain the desired result.
	Note that the convergence
	in
	\cite[Lemma 3.7]{paper_abstract}
	then means
	$R\FF f_n - Q^{-1}(u_n,v_n,q_n) \rightarrow R\FF f - Q^{-1}(u,v,q)$
	in
	$C(\HH)$,
	so that the convergence
	$(u_n,v_n,q_n) \rightarrow (u,v,q)$
	in
	$C(H^1(\O;\Rd)) \times L^1(L^2(\O;\Rd)) \times C(L^2(\O;\Rdds))$
	follows from the fact that
	the range of
	$R$
	is a subset of
	$\{ 0 \} \times L^2(\O;\Rd) \times \{ 0 \}$.
\end{proof}

\begin{proposition}[Strong convergence for fixed forces]
\label{prop:strong_convergence_for_fixed_forces_inertia}
	Let $f \in H^1(L^2(\Omega;\Rd))$ and
	$(u,v,q) \in H^1(\HH)$ be the solution of \cref{eq:state_equation_as_EVI_inertia}
	and $(u_n,v_n,q_n) \in H^1(\HH)$,
	for every $n \in \mathbb{N}$, the solution of
	\begin{align*}
		Q^{-1}(\boverdot{u}_n,\boverdot{v}_n, \boverdot{q}_n)
		+ \AA_{\lambda_n}(u_n,v_n,q_n) &= Rf, \\
		(u_n,v_n,q_n)(0) &= (u_0,v_0,q_0).
	\end{align*}
	where $\sequence{\lambda}{n} \subset (0,\infty)$, $\lambda_n \searrow 0$.
	
	Then $(u_n,v_n,q_n) \rightarrow (u,v,q)$ in $H^1(\HH)$.
\end{proposition}
\begin{proof}
	We can argue as in the proof of 
	\cref{thm:weak_convergence_of_the_state_inertia},
	the assertion follows then directly	from
	\cite[Proposition 3.5]{paper_abstract}.
\end{proof}

\section{Optimal Control}
\label{sec:inertia_oc}

Also in this
section
we will make use of
\cite{paper_abstract}.
Since the smoothed operator
$\AA_s$,
given in
\cref{def:AAs_inertia},
possesses the required properties for
$A_s$
in
\cite[Assumption 5.1 (ii)]{paper_abstract}
with
$\ZZ = \HH$
(as we will see in
\cref{def:AAs_inertia}
and
\cref{prop:AAsFrechet}
below),
we can apply the finding concerned with
the differentiability of the
solution operator associated with
the EVI
therein.
Before we give the details in
\cref{sec:inertia_oc_oc},
we tend to the existence and
approximation of optimal controls.

\subsection{Existence and Approximation of Optimal Controls}
\label{sec:inertia_oc_eaaooc}

Let us now consider
the optimal control problem
\cref{eq:optimization_problem_inertia}.
Note that we assumed in
\cref{sec:2}
that
$\Psi$
is defined on
$L^2(\HH)$
and not on
$H^1(\HH)$,
which excludes for example
evaluations at certain points
in time.
Similar as in
\cite{paper_abstract},
we could also consider an
objective function on
$H^1(\HH)$,
then we would only obtain a
(possible)
weak solution
of the adjoint state
in
\cref{thm:KKT_conditions_inertia},
see also
\cite[Theorem 5.12]{paper_abstract}.
We decided to define
$\Psi$
on
$L^2(\HH)$
only for simplicity and to keep the
discussion concise.

Since we have transformed our state equation
\cref{eq:state_equation_inertia}
into
\cref{eq:state_equation_transformed_inertia}
by introducing the new variable
$q$,
it is reasonable to do the same with the
optimal control problem.
To this end,
we need the following

\begin{definition}[Transformed objective function]
\label{def:transformed_objective_function}
	We define
	\begin{align*}
		\Psi_z : L^2(\HH) \rightarrow \R,
		\largespace
		(u,v,q) \mapsto \Psi(u,v, \ZZZ(u,q))
	\end{align*}
	and the
	\emph{transformed objective function}
	\begin{align*}
		J_z : L^2(\HH) \times \XXX_c \rightarrow \R,
		\largespace
		(u,v,q,f) \mapsto
		\Psi_z(u,v,q)
		+
		\frac{\alpha}{2}\norm{f}{\XXX_c}^2
	\end{align*}
\end{definition}

Using the
definition
above and the transformation
of the state equation into
\cref{eq:state_equation_as_EVI_inertia},
we obtain the equivalence of
\cref{eq:optimization_problem_inertia}
and
\begin{equation}
\label{eq:optimization_problem_inertia_q}
\left\{\mediumspace
\begin{aligned}
	\min \shortspace
	&
	J_z(u, v,q,f) 
	=
	\Psi_z(u,v,q)
	+
	\frac{\alpha}{2}\norm{f}{\XXX_c}^2,
	\\
	\text{ s.t. } \shortspace
	&
	Q^{-1}(\boverdot{u},\boverdot{v}, \boverdot{q}) + \AA(u,v,q) \ni Rf,
	\mediumspace
	(u, v, q)(0) = (u_0, v_0, q_0),
	\\
	&
	(u,v,q) \in H^1(H^1_D(\O;\Rd) \times L^2(\O;\Rd) \times L^2(\O;\Rdds)),
	\\
	&
	f \in \XXX_c.
\end{aligned}\right.
\end{equation}

Let us now select a sequence
$\sequence{\lambda}{n} \subset (0,\infty)$
such that
$\lambda_n \searrow 0$.
We consider the regularized
optimization problem
\begin{equation}
\label{eq:optimization_problem_inertia_n}
\left\{\mediumspace
\begin{aligned}
	\min \shortspace
	&
	J_z(u, v,q,f) 
	=
	\Psi_z(u,v,q)
	+
	\frac{\alpha}{2}\norm{f}{\XXX_c}^2,
	\\
	\text{ s.t. }
	\shortspace
	&
	Q^{-1}(\boverdot{u},\boverdot{v}, \boverdot{q})
	+
	\AA_{\lambda_n}(u,v,q)
	=
	Rf,
	\mediumspace
	(u, v, q)(0) = (u_0, v_0, q_0)
	\\
	&
	(u,v,q) \in H^1(H^1_D(\O;\Rd) \times L^2(\O;\Rd) \times L^2(\O;\Rdds)),
	\\
	&
	f \in \XXX_c.
\end{aligned}\right.
\end{equation}

\begin{theorem}[Existence and approximation of optimal solutions]
\label{thm:existence_and_approximation_of_optimal_solutions_inertia}
	Suppose that the control space
	$\XXX_c$
	is such that
	$\FF : H^1(L^2(\O;\Rd)) \to H^2(L^2(\O;\Rd))$
	with
	$\FF(f)(t) = \int_0^t f(s) ds$
	is compact from
	$\XXX_c$
	into
	$L^1(L^2(\O;\Rd))$.

	Then there exists a global solution of
	\cref{eq:optimization_problem_inertia_q}
	(and thus of \cref{eq:optimization_problem_inertia})
	and of
	\cref{eq:optimization_problem_inertia_n}
	for every
	$n \in \N$. 
	
	Moreover, let
	$(\overline{u}_n,\overline{v}_n,\overline{q}_n,
	\overline{f}_n)_{n \in \N}$
	be a sequence of global solution of
	\cref{eq:optimization_problem_inertia_n}.
	Then there exists a weak
	accumulation point
	$(\overline{u},\overline{v},\overline{q},\overline{f})$
	and every weak accumulation point
	is a global solution of
	\cref{eq:optimization_problem_inertia_q}.
	The subsequence of states
	which converges weakly towards
	$(\overline{u},\overline{v},\overline{q})$
	in
	$H^1(\HH)$,
	converges also strongly in
	$C(H^1(\O;\Rd)) \times L^1(L^2(\O;\Rd)) \times C(L^2(\O;\Rdds))$
	and,
	when
	$\FF$
	is compact from
	$\XXX_c$
	into
	$C(L^2(\O;\Rd))$,
	then the subsequence of
	$v_n$
	converges also strongly in
	$C(L^2(\O;\Rdds))$.
	Moreover,
	the subsequence of
	controls converges strongly to
	$\overline{f}$
	in
	$\XXX_c$.
\end{theorem}
\begin{proof}
	The existence of
	a global solution to
	\cref{eq:optimization_problem_inertia_q}
	follows from the standard
	direct method of the calculus of variations using
	\cref{thm:weak_convergence_of_the_state_inertia}
	and the assumed compactness of
	$\FF$,
	the proof is for instance
	analog to the proof of
	\cite[Theorem 4.2]{paper_abstract}.
	The existence of a global solution to
	\cref{eq:optimization_problem_inertia_n}
	follows easily using the
	Lipschitz continuity of
	$\AA_{\lambda_n}$
	(which implies the
	Lipschitz continuity
	of the corresponding
	solution operator).
	
	The convergence result can also
	be obtained by standard arguments
	using again
	\cref{thm:weak_convergence_of_the_state_inertia}
	and
	\cref{prop:strong_convergence_for_fixed_forces_inertia},
	the proof is again analog to
	\cite[Theorem 4.5 \& Corollary 4.6]{paper_abstract}.
	Note that the strong convergence
	of the states in
	$C(H^1(\O;\Rd)) \times L^1(L^2(\O;\Rd)) \times C(L^2(\O;\Rdds))$,
	and also of
	$v_n$
	in
	$C(L^2(\O;\Rdds))$
	when
	$\FF$
	is compact from
	$\XXX_c$
	into
	$C(L^2(\O;\Rd))$,
	follows directly from
	\cref{thm:weak_convergence_of_the_state_inertia}.
\end{proof}

A strong convergence result
of the states
(in
$H^1(\HH)$)
is not provided in the
theorem
above.
In
\cite[Corollary 4.6]{paper_abstract}
we were able to prove the strong convergence
either when the associated maximal monotone operator
is a subdifferential,
which is here not the case
(\cref{rem:AA_is_not_a_subdifferential_inertia}),
or when it can be deduced from the weak convergence
and the convergence of the evaluations of
$\Psi$.
Since we supposed that
$\Psi$
is defined on
$L^2(\HH)$,
this cannot be the case.
However,
as elaborated on at the beginning of this
section,
it is possible for instance to consider
a different
$\Psi$
defined on
$H^1(\HH)$
such that this property holds.

Let us shortly interrupt the discussion
and give two examples for the control space
$\XXX_c$.

\begin{example}[Control space]
\label{ex:control_space_inertia}
	In order to satisfy
	the assumption on
	$\XXX_c$
	in	
	\cref{thm:existence_and_approximation_of_optimal_solutions_inertia},
	we can use the lemma of Lions-Aubin
	(cf.
	\cite[III. Proposition 1.3]{showalter})
	and for instance choose
	$\XXX_c = H^1(L^2(\O;\Rd)) \cap L^2(H^1(\O;\Rd))$
	or
	$\XXX_c = \{
		f \in H^1(L^2(\O;\Rd))
		:
		\FF f \in L^2(H^1(\O;\Rd))
	\}$
	with corresponding norms.
\end{example}

Having dealt with the existence and approximation
of optimal solutions we
turn to the optimality condition for a
further smoothed problem.

\subsection{Optimality Conditions}
\label{sec:inertia_oc_oc}

In order to derive
first order optimality conditions
we
smoothen at first the optimal control
problem further.
Then we prove the differentiability
of the smoothed solution operator
and can after that finally present our main result,
the optimality
conditions for the
smoothed optimization problem.

We impose the following assumptions
for the rest of this subsection.

\begin{assumption}[Standing assumptions for \cref{sec:inertia_oc_oc}]
\label{assu:standing_assumptions_for_sec42}
	\begin{enumerate}
	[label=(\roman*)]
		\item
		\label{assu:item:Rs_inertia}
		Let
		$R_s : \Rdds \rightarrow \Rdds$
		be monotone,
		Lipschitz continuous and
		Fr\'{e}chet differentiable.
		
		\item
		\label{assu:item:phat_and_p_inertia}
		We fix
		$2 < \hat{p} < p < \overline{p}$,
		where
		$\overline{p}$
		is from
		\cref{cor:W1r_existence_inertia}
		(respectively
		\cite[Theorem 1.1]{herzog}),
		such that
		$2 - \frac{d}{2} \geq - \frac{d}{\overline{p}}$.
		
		\item
		\label{assu:item:initial_conditions_regularity_inertia}
		Let the initial data
		$(u_0, v_0, q_0)$
		be an element of
		$\YY_p$,
		where
		$p$
		is given in
		\Cref{assu:item:phat_and_p_inertia}.
	\end{enumerate}	 
\end{assumption}

Thanks to
\cref{prop:concrete_form_of_yosida_inertia},
we can give the precise form
of the resolvent
and Yosida approximation of
$\AA$
in the following

\begin{corollary}[Precise form of the resolvent]
\label{cor:concrete_form_of_yosida_inertia}
	Let
	$\lambda > 0$
	and denote the resolvent of
	$\AA$
	by
	$\RR_\lambda$.
	Then
	\begin{align*}
		\RR_\lambda(h)
		=
		\begin{pmatrix}
			\TT_{R_\lambda}(h) \\
			\frac{1}{\lambda} \TT_{R_\lambda}(h) - \frac{h_1}{\lambda} \\
			R_\lambda(\E^{\top} \symnabla(\TT_{R_\lambda}(h) - h_1) + h_3)
		\end{pmatrix}
	\end{align*}
	so that
	\begin{align*}
		\AA_\lambda(h)
		=
		\frac{1}{\lambda}
		\begin{pmatrix}
			h_1 - \TT_{R_\lambda}(h) \\
			h_2 - \frac{1}{\lambda} \TT_{R_\lambda}(h) + \frac{h_1}{\lambda} \\
			h_3 - R_\lambda(\E^{\top} \symnabla(\TT_{R_\lambda}(h) - h_1) + h_3)
		\end{pmatrix}
	\end{align*}
	for every
	$h = (h_1,h_2,h_3) \in \HH$.
\end{corollary}

The Yosida approximation
$\AA_\lambda$
is in view of
\cref{prop:TR_lipschitz_inertia}
Lipschitz continuous from
$\YY_p$
to
$\YY_p$,
where
$p$
is given in
\cref{assu:standing_assumptions_for_sec42}
\cref{assu:item:phat_and_p_inertia}.
Therefore the state equation in
\cref{eq:optimization_problem_inertia_n}
admits a solution in
$\YY_p$
(note that
$R$
maps into
$\YY_\infty$).
However,
since this regularity is not present in
\cref{eq:optimization_problem_inertia},
we did not use it.
In contrast,
the same is true for
the smoothed Yosida approximation,
which is given below in
\cref{def:AAs_inertia}
(see
\cref{def:smoothed_solution_operator_inertia}),
but here this additional regularity
will be used to prove the differentiability
of the smoothed solution operator in
\cref{prop:frechet_differentiability_of_the_smoothed_solution_operator_inertia}.

In order to smoothen
the Yosida approximation,
respectively the resolvent,
of
$\AA$,
we smoothen the resolvent of
$A$
and then define the smoothed resolvent for
$\AA$
analog to
$\RR_\lambda$.
We denote this smoothed resolvent of
$A$
by
$R_s : \Rdds \to \Rdds$
(which indicates that
the resolvent of
$A$
can be expressed pointwise),
from the properties given in
\cref{assu:standing_assumptions_for_sec42}
\cref{assu:item:Rs_inertia}
one can easily derive
the following inequalities,
which will be useful when
proving the differentiability of
$\TT_{R_s}$
in
\cref{lem:T_Frechet_inertia}
below.

\begin{lemma}[Properties of $R_s'$]
\label{lem:Rs_derivative_estimate_inertia}
	There exists a constant
	$C$
	such that
	$| R_s'(\sigma) \tau |
	\leq C | \tau |$
	and
	$0 \leq R_s'(\sigma)\tau \colon \tau$
	holds for all
	$\sigma,\tau\in\Rdds$.
	Moreover,
	the same is true for
	$R_s'(\cdot)^*$.
\end{lemma}
\begin{proof}
	Let
	$\sigma,\tau\in\Rdds$
	be arbitrary.
	The Lipschitz continuity and
	Fr\'{e}chet differentiability of
	$R_s$
	gives
	\begin{align*}
		\Big{|}
			\frac{r(t\tau)}{t}
			+
			R_s'(\sigma)\tau
		\Big{|}
		=
		\frac{| R_s(\sigma + t\tau) - R_s(\sigma) |}{t}
		\leq
		L | \tau |
	\end{align*}
	for all $t \in \R\setminus\{ 0 \}$,
	where
	$r$
	is the remainder term of
	$R_s$.
	The limit $t \rightarrow 0$ yields the first assertion.
	
	The second claim follows using the monotonicity,
	\begin{align*}
		0
		\leq
		\frac{R_s(\sigma + t\tau) - R_s(\sigma)}{t} \colon \tau
		\rightarrow
		R_s'(\sigma)\tau \colon \tau
	\end{align*}
	as
	$0 \neq t \rightarrow 0$.
	
	Now,
	by definition we have
	$R_s'(\sigma)\tau \colon \eta = \tau \colon R_s'(\sigma)^*\eta$
	for all
	$\sigma,\tau,\eta \in \Rdds$,
	so that the second assertion also holds for
	$R_s'(\cdot)^*$.
	Choosing in particular
	$\tau = R_s(\sigma)^*\eta$
	we get
	\begin{align*}
		| R_s'(\sigma)^*\eta |^2
		=
		| R_s'(\sigma)R_s'(\sigma)^*\eta \colon \eta |
		\leq
		C
		| R_s'(\sigma)^*\eta |
		\
		| \eta |,
	\end{align*}
	which yields the first assertion for
	$R_s'(\cdot)^*$.
\end{proof}

\begin{definition}[Smoothed resolvent]
\label{def:AAs_inertia}
	Let
	$\lambda_s \in (0,\infty)$.
	We define
	\begin{align*}
		\RR_s : \YY_p \rightarrow \YY_p,
		\mediumspace
		h = (h_1,h_2,h_3) \mapsto
		\begin{pmatrix}
			\TT_{R_s}(h) \\
			\frac{1}{\lambda_s} \TT_{R_s}(h) - \frac{h_1}{\lambda_s} \\
			R_s(\E^{\top} \symnabla (\TT_{R_s}(h) - h_1) + h_3)
		\end{pmatrix}
	\end{align*}
	and
	$\AA_s \define \frac{1}{\lambda_s}(I - \RR_s)$
	(see
	\cref{assu:standing_assumptions_for_sec42}
	\cref{assu:item:phat_and_p_inertia}
	for
	$p$).
	According to
	\cref{prop:TR_lipschitz_inertia}
	and
	\cref{assu:standing_assumptions_for_sec42}
	\cref{assu:item:Rs_inertia},
	$\RR_s$
	and
	$\AA_s$
	are well defined and Lipschitz continuous.
	As usual,
	with a slight abuse of notation,
	we denote operators for different
	$p$
	with the same symbol.
\end{definition}

Let us now consider
the smoothed optimization problem
\begin{equation}
\label{eq:optimization_problem_inertia_s}
\left\{\mediumspace
\begin{aligned}
	\min \shortspace
	&
	J_z(u, v,q,f) 
	=
	\Psi_z(u,v,q)
	+
	\frac{\alpha}{2}\norm{f}{\XXX_c}^2,
	\\
	\text{ s.t. }
	\shortspace
	&
	Q^{-1}(\boverdot{u},\boverdot{v}, \boverdot{q})
	+
	\AA_{s}(u,v,q)
	=
	Rf,
	\mediumspace
	(u, v, q)(0) = (u_0, v_0, q_0)
	\\
	&
	(u,v,q) \in H^1(H^1_D(\O;\Rd) \times L^2(\O;\Rd) \times L^2(\O;\Rdds)),
	\\
	&
	f \in \XXX_c.
\end{aligned}\right.
\end{equation}

Analog to
\cref{thm:existence_and_approximation_of_optimal_solutions_inertia}
one can analogously prove
that there exists
a global solution of
\cref{eq:optimization_problem_inertia_s}.

As was done in
\cite[Theorem 4.5 \& Corollary 4.6]{paper_abstract},
when
$\AA_s$
and
$\AA_{\lambda_s}$
are globally
``close together'',
one can prove a result analog to
the convergence result in
\cref{thm:existence_and_approximation_of_optimal_solutions_inertia}
with a sequence
$(\overline{u}_s,\overline{v}_s,\overline{q}_s,
\overline{f}_s)_{s > 0}$
of global solutions to
\cref{eq:optimization_problem_inertia_s}
when
$\sup_{h\in\HH}\norm{\AA_{\lambda_s}(h) - \AA_s(h)}{\HH}$
tends fast enough to zero relative to
$\lambda_s$.
The following lemma shows
that this is the case
when the same is true for
$\frac{1}{\lambda_s}\sup_{\tau \in L^2(\O;\Rdds)}
\norm{A_{\lambda_s}(\tau) - A_s(\tau)}{L^2(\O;\Rdds)}$
with
$A_s = \frac{1}{\lambda_s}(I - R_s)$,
which holds in the case of
the von-Mises flow rule
investigated in
\cref{sec:examples}
below for suitable sequences
$\{ \lambda_s \}_{\lambda_s > 0}$
and
$\{ s \}_{s > 0}$,
cf.
\cref{eq:smoothYosidaEstimate}.
Note also that
\cite[Lemma 3.15]{paper_abstract}
was used in
\cite[Theorem 4.5 \& Corollary 4.6]{paper_abstract},
so that it was in particular required that
$\sup_{\tau \in L^2(\O;\Rdds)}
\norm{A_{\lambda_s}(\tau) - A_s(\tau)}{L^2(\O;\Rdds)}$
tends faster to zero than
$\textup{exp}(\frac{1}{\lambda_s})$,
thus the additional factor
$\frac{1}{\lambda_s}$
does not play a big role.

\begin{lemma}[Convergence of the smoothed resolvent]
\label{lem:estimate_for_convergence_of_the_regularized_yosida_approximation_inertia}
	The inequality
	\begin{align*}
		\norm{\AA_{\lambda_s}(h) - \AA_s(h)}{\HH}
		\leq
		C \sqrt{1 + \frac{1}{\lambda_s^2}}
		\sup_{\tau \in L^2(\O;\Rdds)}
		\norm{A_{\lambda_s}(\tau) - A_s(\tau)}{L^2(\O;\Rdds)}
	\end{align*}
	holds for all
	$h \in L^2(\O;\Rdds)$,
	where
	$A_s \define \frac{1}{\lambda_s}(I - R_s)$
	and the constant
	does only depend on
	$\C$
	and
	$\B$,
	$C = C(\C, \B)$.
\end{lemma}
\begin{proof}
	Let us abbreviate
	\begin{align*}
		M \define
		\sup_{\tau \in L^2(\O;\Rdds)}
		\norm{R_{\lambda_s}(\tau) - R_s(\tau)}{L^2(\O;\Rdds)}.
	\end{align*}
	Due to the definitions of
	$\AA_s$
	and
	$A_s$
	we only have to prove that
	\begin{align}
	\label{eq:local028102019}
		\norm{\RR_{\lambda_s}(h) - \RR_s(h)}{\HH}
		\leq
		C \sqrt{1 + \frac{1}{\lambda_s^2}} M
	\end{align}
	holds for all
	$h \in \HH$.
	To this end let
	$h \in \HH$
	be arbitrary and abbreviate
	$u \define \TT_{R_{\lambda_s}}(h),
	u_s \define \TT_{R_s}(h) \in H^1_D(\O;\Rd)$,
	hence,
	$u$
	is the solution of
	\cref{eq:TR_def_inertia}
	with respect to
	$R_{\lambda_s}$
	and
	$u_s$ with respect to
	$R_s$,
	testing both equations with
	$u - u_s$
	and subtracting the second
	from the first,
	we get
	\begin{align*}
		&
		\norm{\symnabla (u - u_s)}{L^2(\O;\Rdds)_{\D}}^2
		+
		\bignorm{\frac{u - u_s}{\lambda_s}}{L^2(\O;\Rd)}^2
		\\
		&\largespace\largespace
		=
		-\scalarproduct{\E(R_{\lambda_s}(w) - R_{s}(w_s))}
		{\symnabla (u - u_s)}{L^2(\O;\Rdds)}
		\\
		&\largespace\largespace
		=
		-\scalarproduct{(R_{\lambda_s}(w_s) - R_{s}(w_s))}
		{w - w_s}{L^2(\O;\Rdds)}
		\\
		&\largespace\largespace\largespace\largespace
		-\scalarproduct{(R_{\lambda_s}(w) - R_{\lambda_s}(w_s))}
		{w - w_s}{L^2(\O;\Rdds)}
		\\
		&\largespace\largespace
		\leq
		-\scalarproduct{\D^{-1}\E(R_{\lambda_s}(w_s) - R_{s}(w_s))}
		{\symnabla (u - u_s)}{L^2(\O;\Rdds)_{\D}}
		\\
		&\largespace\largespace
		\leq
		\frac{1}{2}
		\norm{\D^{-1}\E(R_{\lambda_s}(w) - R_{s}(w_s))}
		{L^2(\O;\Rdds)_{\D}}^2
		+
		\frac{1}{2}
		\norm{\symnabla (u - u_s)}{L^2(\O;\Rdds)_{\D}}^2
		\\
		&\largespace\largespace
		\leq
		\frac{\norm{\E^{\top} \D^{-1}\E}{}}{2} M^2
		+
		\frac{1}{2}
		\norm{\symnabla (u - u_s)}{L^2(\O;\Rdds)_{\D}}^2
	\end{align*}
	with
	$w \define \E^{\top} \symnabla (u - h_1) + h_3)$
	and
	$w_s \define \E^{\top} \symnabla(u_s - h_1) + h_3)$,
	where we used in particular the monotonicity of
	$R_{\lambda_s}$.
	Thus we obtain
	\begin{align}
	\label{eq:local18112020}
		\norm{\symnabla (u - u_s)}{L^2(\O;\Rdds)_{\D}}^2
		+
		\bignorm{\frac{u - u_s}{\lambda_s}}{L^2(\O;\Rd)}^2
		\leq
		C M^2.
	\end{align}
	We get further
	\begin{align*}
		\norm{R_{\lambda_s}(w) &- R_s(w_s)}{L^2(\O;\Rdds)}
		\\
		&\leq
		\norm{R_{\lambda_s}(w) - R_{\lambda_s}(w_s)}{L^2(\O;\Rdds)}
		+
		\norm{R_{\lambda_s}(w_s) - R_s(w_s)}{L^2(\O;\Rdds)}
		\\
		&\leq
		\frac{C}{\lambda_s}
		M
		+
		M		
	\end{align*}
	where we have used
	\cref{eq:local18112020}.
	We arrive at
	\begin{align*}
		\norm{\RR_{\lambda_s}(h) - \RR_s(h)}{\HH}^2
		\leq
		C M^2
		+
		\frac{C}{\lambda_s^2}M^2
	\end{align*}
	which implies
	\cref{eq:local028102019}.
\end{proof}

Let us now turn to optimality conditions.
We first need to prove
the Fr\'{e}chet differentiability of
the smoothed solution operator
of the constraint in
\cref{eq:optimization_problem_inertia_s}.
To this end,
we need two norm gaps in
\cref{lem:T_Frechet_inertia}
and
\cref{prop:AAsFrechet},
recall that the corresponding coefficients
are fixed in
\cref{assu:standing_assumptions_for_sec42}
\cref{assu:item:phat_and_p_inertia}.

\begin{lemma}[Fr\'{e}chet differentiability of $\TT_{R_s}$]
\label{lem:T_Frechet_inertia}
	The operator
	$\TT_{R_s}$
	is from
	$\YY_{p}$
	into
	$W^{1,\hat{p}}_D(\O;\Rd)$
	Fr\'{e}chet differentiable and,
	for
	$h,g \in  \YY_p$,
	$\eta \define \TT_{R_s}'(h) g$
	is of class
	$W^{1,p}_D(\O;\Rd)$
	and the unique solution of
	\begin{align}
	\label{eq:TDerivative}
	\begin{split}
		&
		-\div(
			\D \symnabla \eta
			+
			\E R_s'(\E^{\top} \symnabla (u - h_1) + h_3)
			(\E^{\top} \symnabla (\eta - g_1) + g_3))
		)
		=
		\frac{g_2}{\lambda_s} + \frac{g_1 - \eta}{\lambda_s^2}
	\end{split}
	\end{align}
	for all
	$\varphi \in W^{1,p'}_D(\O;\Rd)$,
	where
	$u \define \TT_{R_s}(h)$.
	
	Moreover,
	there exists a constant
	$C$
	such that
	the extension of
	$\TT_{R_s}'(h)$
	to an element of
	$\LL(\HH;H^1_D(\O;\Rd))$
	fulfills
	$\norm{\TT_{R_s}'(h)g}{H^1_D(\O;\Rd)}
	\leq
	C \norm{g}{\HH}$
	for all
	$h \in \YY_p$
	and
	$g \in \HH$.
\end{lemma}
\begin{proof}
	Let
	$h,g \in \YY_p$.
	At first we prove that
	\cref{eq:TDerivative}
	has a unique solution
	$\eta \in W^{1,p}_D(\O;\Rd)$
	with respect to
	$h$
	and
	$g$.
	For
	$\sigma \in \Rdds$
	we define
	$b_\sigma : \O \times \Rdds \rightarrow \Rdds$
	by
	\begin{align*}
		b_\sigma(x,\tau)
		\define \D \tau
		+
		\E R_s'(\E^{\top} \symnabla (u(x) - h_1(x)) + h_3(x)))
		(\E^{\top} \tau + \sigma)
	\end{align*}
	for almost all
	$x \in \O$
	and all
	$\tau \in \Rdds$.
	The existence of
	$\eta$
	follows now from
	\cref{cor:W1r_existence_inertia}
	(with
	$\varphi \define -\E^{\top} \symnabla g_1 + g_3$),
	when we have verified the requirements on
	$b_\sigma$
	therein.
	Moreover,
	\cref{cor:W1r_existence_inertia}
	also shows that
	the solution operator of
	\cref{eq:TDerivative}
	is continuous with respect to
	$g \in \YY_p$
	(clearly, it is also linear).
	
	Clearly,
	$b_{0}(x,0) = 0 \in L^\infty(\O; \Rdds)$
	and
	$b_{\sigma}(\cdot,\tau)$
	is measurable as a pointwise limit of
	measurable functions
	(see
	\cite[Corollary 3.1.5]{wachsmuth}),
	for all
	$\tau, \sigma \in \Rdds$.
	Moreover,
	we have
	\begin{align*}
		(b_\sigma(x,\tau)
		&
		-
		b_{\overline{\sigma}}(x,\overline{\tau})) \colon
		(\tau - \overline{\tau})
		\\
		&\geq
		\gamma_{\D} | \tau - \overline{\tau} |^2
		+
		R_s'(w(x)) (\E^{\top} (\tau - \overline{\tau})
		+ (\sigma - \overline{\sigma}))
		\colon
		\E^{\top} (\tau - \overline{\tau})
		\\
		&\geq
		\gamma_{\D}
		| \tau - \overline{\tau} |^2
		-
		C | \sigma - \overline{\sigma} | \
		| \tau - \overline{\tau} |,
	\end{align*}
	with
	$w \define \E^{\top} \symnabla (u - h) + h_3)$,
	and
	\begin{align*}
		| b_\sigma(x,\tau)
		&
		- b_{\overline{\sigma}}(x,\overline{\tau}) |
		\leq
		C \Big{(}
			| \tau - \overline{\tau} |
			+
			| \sigma - \overline{\sigma} |
		\Big{)}
	\end{align*}
	for all
	$\sigma,\overline{\sigma},\tau,\overline{\tau} \in \Rdds$
	and almost all
	$x \in \O$,
	where we have used
	\cref{lem:Rs_derivative_estimate_inertia}
	in both estimations.
	Therefore
	\crefrange{eq:W1r_existence_inertia_1}{eq:W1r_existence_inertia_4}
	are fulfilled.
	
	Considering now the equations for
	$u_g \define \TT_{R_s}(h + g)$
	and
	$u \define \TT_{R_s}(h)$,
	we see that
	\begin{align*}
		-\div(
			\D \symnabla (u_g - u - \eta)
		)
		+
		\frac{u_g - u - \eta}{\lambda_s^2}
		&=
		\div(
			\E(R_s(\mu + \nu_g) - R_s(\mu) - R_s'(\mu) \nu_g)
		)
		\\
		&\largespace
		+\div(
			\E R_s'(\mu) ((\E^{\top} \symnabla(u_g - u - \eta))
		),
	\end{align*}
	where
	\begin{align*}
		\mu &\define \E^{\top} \symnabla (u - h_1) + h_3),
		\\
		\nu_g &\define \E^{\top} \symnabla (u_g - u - g_1) + g_3) \in L^p(\O;\Rdds),
	\end{align*}
	hence,
	\begin{align*}
		&
		-\div(
			\D \symnabla (u_g - u - \eta)
			-
			\E R_s'(\mu) ((\E^{\top} \symnabla(u_g - u - \eta))
		)
		+
		\frac{u_g - u - \eta}{\lambda_s^2}
		=
		\div \E r_{\mu}(\nu_g),
	\end{align*}
	where
	$r_{\mu}(\nu_g)$
	is the remainder term of
	$R_s$
	at
	$\mu$
	in direction
	$\nu_g$.
	Applying
	\cref{cor:W1r_existence_inertia}
	once again with
	\begin{align*}
		b_\sigma(x,\tau)
		\define
		\D \tau
		+
		\E R_s'(\mu(x))\E^{\top} \tau
	\end{align*}
	(and
	$p = \hat{p}$)
	we obtain
	\begin{align*}
		\frac{\norm{u_g - u - \eta}{W^{1,\hat{p}}(\O;\Rd)}}
		{\norm{g}{\YY_p}}
		\leq
		C \frac{\norm{r_{\mu}(\nu_g)}{L^{\hat{p}}(\O;\Rdds)}}
		{\norm{g}{\YY_p}}
		\leq
		C \frac{\norm{r_{\mu}(\nu_g)}{L^{\hat{p}}(\O;\Rdds)}}
		{\norm{\nu_g}{L^p(\O;\Rdds)}}
		\rightarrow 0,
	\end{align*}
	as
	$g \rightarrow 0$
	in
	$\YY_p$,
	where we also used the Lipschitz continuity of
	$\TT_{R_s}$
	and the fact that
	$R_s : L^p(\O;\Rdds) \rightarrow L^{\hat{p}}(\O;\Rdds)$
	is Fr\'{e}chet differentiable
	(cf.
	\cite[Theorem 7]{goldberg}).
	
	That
	the extension of
	$\TT_{R_s}'(h)$
	to an element of
	$L(\HH;H^1_D(\O;\Rd))$
	fulfills the asserted inequality,
	can be proven as above
	(one can simply test
	\cref{eq:TDerivative}
	with
	$\eta \in H^1_D(\O;\Rd)$
	and use
	\cref{lem:Rs_derivative_estimate_inertia}).
\end{proof}

\begin{proposition}
[Fr\'{e}chet differentiability of $\AA_s$]
\label{prop:AAsFrechet}
	The mapping
	$\RR_s$
	is from
	$\YY_p$
	to
	$\HH$
	Fr\'{e}chet differentiable
	and there exists a constant
	$C$
	such that
	the extension of
	$\RR_s'(h) \in \LL(\YY_p; \HH)$
	to an element of
	$\LL(\HH)$
	fulfills
	$\norm{\RR_s'(h)g}{\HH} \leq C \norm{g}{\HH}$
	for all
	$h \in \YY_p$
	and
	$g \in \HH$.
	
	For
	$h \in \YY_p$
	and
	$g \in \HH$
	we have
	\begin{align*}
		\RR_s'(h)g
		=
		\begin{pmatrix}
			\TT_{R_s}'(h)g \\
			\frac{1}{\lambda_s} \TT_{R_s}'(h)g - \frac{g_1}{\lambda_s} \\
			R_s'(\E^{\top} \symnabla(\TT_{R_s}(h) - h_1) + h_3)
			(\E^{\top} \symnabla(\TT_{R_s}'(h)g - g_1) + g_3)
		\end{pmatrix}
	\end{align*}
	The same is true for
	$\AA_s = \frac{1}{\lambda_s} (I - \RR_s)$
	with
	$\AA_s'(h)g = \frac{1}{\lambda_s} (g - \RR_s'(h)g)$
	for all
	$h \in \YY_p$
	and
	$g \in \HH$.
\end{proposition}
\begin{proof}
	The assertion follows from
	\cref{lem:T_Frechet_inertia},
	\cref{lem:Rs_derivative_estimate_inertia}
	for the estimate of
	$(\RR_s'(h)g)_3$, 
	the fact that
	$R_s : L^{\hat{p}}(\O;\Rdds) \rightarrow L^2(\O;\Rdds)$
	is Fr\'{e}chet differentiable
	(cf.
	\cite[Theorem 7]{goldberg})	
	and the chain rule.
\end{proof}

Now,
we can use
\cite[Theorem 5.5]{paper_abstract}
to derive the differentiability
of the solution operator of the constraint in
\cref{eq:optimization_problem_inertia_s}
from the differentiability of
$\AA_s$.
To this end,
we first introduce
the solution operator in

\begin{definition}[Smoothed solution operator]
\label{def:smoothed_solution_operator_inertia}
	We denote the solution operator of
	\begin{align}
	\label{eq:smoothed_state_equation_inertia}
		Q^{-1}(\boverdot{u},\boverdot{v}, \boverdot{q})
		+
		\AA_{s}(u,v,q)
		=
		Rf,
		\mediumspace
		(u, v, q)(0) = (u_0, v_0, q_0)
	\end{align}
	by
	$\SS_s : L^2(L^2(\O;\Rd))
	\rightarrow H^1(\YY_p)$,
	that is,
	$\SS_s(f) = (u,v,q)$,
	which existence follows from
	Banachs contraction principle
	since
	$\AA_s$
	is Lipschitz continuous
	according to
	\cref{def:AAs_inertia}.
	Here we use the improved
	regularity of
	$(u_0, v_0, q_0)$,
	see
	\cref{assu:standing_assumptions_for_sec42}
	\cref{assu:item:initial_conditions_regularity_inertia}.
\end{definition}

\begin{proposition}[Fr\'{e}chet differentiability of the smoothed solution operator]
\label{prop:frechet_differentiability_of_the_smoothed_solution_operator_inertia}
	The solution operator
	$\SS_s : L^2(L^2(\O;\Rd)) \rightarrow H^1(\YY_p)$
	is Lipschitz continuous,
	$\SS_s : H^1(L^2(\O;\Rd))
	\rightarrow H^1(\HH)$
	is Fr\'{e}chet differentiable and,
	for
	$f,g \in H^1(L^2(\O;\Rd))$,
	$\eta \define \SS_s'(f)g \in H^1(\HH)$
	is the unique solution of
	\begin{align}
	\label{eq:eta_equation_inertia}
		Q^{-1}\boverdot{\eta} + \AA_s'(w)\eta = Rg,
		\mediumspace
		\eta(0) = 0,
	\end{align}
	where
	$w \define \SS_s(f)$.
	Moreover,
	there exists a constant
	$C$,
	such that
	$\norm{\SS_s'(f)g}{H^1(\HH)}
	\leq	
	C \norm{g}{L^2(L^2(\O;\Rd))}$
	holds for all
	$f,g \in H^1(L^2(\O;\Rd))$.
\end{proposition}
\begin{proof}
	Our goal is to use
	\cite[Theorem 5.5]{paper_abstract},
	to this end we first consider the transformed
	equation from
	\cref{sec:inertia_se_racr}.
	We again set
	$p_0
	\define
	-Q^{-1}(u_0, v_0, q_0)$
	and denote the solution operator of
	\begin{align}
	\label{eq:local009042020}
		\boverdot{p} = \AA_s(RF - Qp), \mediumspace p(0) = p_0
	\end{align}
	by
	$\tilde{\SS}_s : L^2(L^2(\O;\Rd)) \rightarrow H^1(\YY_p)$,
	that is,
	$\tilde{\SS}_s(F) = p$.
	Thus we have
	$\SS_s(f)
	= R\mathcal{F}_\rho f
	- Q \tilde{\SS}_s(\mathcal{F}_\rho f)$
	for all
	$f \in L^2(L^2(\O;\Rd))$.
	
	We can now apply
	\cite[Lemma 5.3 \& Theorem 5.5]{paper_abstract}
	(with
	$\XX = L^2(\O;\Rd)$,
	$\YY =\YY_p$,
	$\ZZ = \HH$,
	$z = p$
	and
	$z_0 = p_0$),
	note that
	the assumptions in
	\cite[Assumption 5.1 (ii)]{paper_abstract}
	are satisfied thanks to
	\cref{prop:AAsFrechet}.
	Thus the solution operator
	$\tilde{\SS}_s : L^2(L^2(\O;\Rd))
	\rightarrow H^1(\YY_p)$
	is Lipschitz continuous and
	$\tilde{\SS}_s : H^1(L^2(\O;\Rd))
	\rightarrow H^1(\HH)$
	is Fr\'{e}chet differentiable,
	hence,
	the desired Lipschitz continuity and
	Fr\'{e}chet differentiability also hold for
	$\SS_s$.
	Furthermore,
	the asserted
	inequality holds and we have
	$\eta = R\mathcal{F}_\rho g - Q \tilde{\eta}$,
	where
	$\eta \define \SS_s'(f)g$
	and
	$\tilde{\eta}
	\define \tilde{\SS}_s'(\mathcal{F}_\rho f)\mathcal{F}_\rho g$.
	\cite[Theorem 5.5]{paper_abstract}
	also shows that
	$\tilde{\eta}$
	is the unique solution of
	\begin{align*}
		\partial_t \tilde{\eta}
		=
		\AA'_s(R\mathcal{F}_\rho f - Qp)
		(R\mathcal{F}_\rho g - Q\tilde{\eta}),
		\mediumspace
		\tilde{\eta}(0) = 0,
	\end{align*}
	where
	$p \define \tilde{\SS}_s(\mathcal{F}_\rho f)$.
	Taking into account that
	$\tilde{\eta} = R\mathcal{F}g - Q^{-1} \eta$
	and
	$\partial_t \mathcal{F}g = g$,
	we see that
	$\eta$
	is the solution of
	\cref{eq:eta_equation_inertia}.
\end{proof}

\begin{remark}[Control space]
	As seen in the
	proposition
	above,
	the smoothed solution operator
	defined on
	$H^1(L^2(\O;\Rd))$
	is Fr\'echet differentiable.
	The norm gaps,
	which arise from the exponents in
	\cref{assu:standing_assumptions_for_sec42}
	\cref{assu:item:phat_and_p_inertia},
	are only needed for the differentiability of
	$\TT_{R_s}$
	but not in the control space.
	Unfortunately,
	we still require the compactness
	property imposed on
	$\XXX_c$
	in
	\cref{thm:existence_and_approximation_of_optimal_solutions_inertia}
	to use the convergence results in
	\cref{sec:inertia_se_racr}.
	However,
	we can avoid taking a subspace of
	$H^1(L^{\tilde{p}}(\O;\Rd))$,
	for a certain
	$\tilde{p} > 2$,
	as the control space.
\end{remark}

Let us now consider
the following
reduced optimization problem
\begin{align}
\label{eq:optimization_problem_r_inertia}
	\min_{f \in \XXX_c} F_z(f),
\end{align}
where the
reduced objective function
$F_z : \XXX_c \rightarrow \R$
is defined by
$F_z(f) \define J_z(\SS_s(f),f)$.
Clearly,
\cref{eq:optimization_problem_r_inertia}
and
\cref{eq:optimization_problem_inertia_s}
are equivalent.

We can finally present our main result.

\begin{theorem}
[Optimality conditions for \cref{eq:optimization_problem_r_inertia}]
\label{thm:KKT_conditions_inertia}
	Let
	$\overline{f} \in \XXX_c$
	and abbreviate
	$(\overline{u},\overline{v},\overline{q})
	\define
	\SS_s(\overline{f})
	\in H^1(\YY_p)$
	and
	$\overline{w}
	\define \TT_{R_s}(\overline{u},\overline{v},\overline{q})
	\in H^1(W^{1,p}_D(\O;\Rdds))$.
	Then the variational equation
	\begin{align}
	\label{eq:Fz_derivative_inertia}
		F_z'(\overline{f})g
		=
		\Psi_{z}'(\SS_s(\overline{f}))\SS_s'(\overline{f})g
		+
		\alpha \scalarproduct{\overline{f}}{g}{\XXX_c} = 0
	\end{align}
	holds for all
	$g \in \XXX_c$
	if and only if
	there exists an unique adjoint state
	$(\varphi, \eta^*)
	=	
	(\varphi_1, \varphi_2, \varphi_3, \eta^*)
	\in H^1(\HH \times H^1_D(\O;\Rd))$
	such that the following
	optimality system is satisfied:
	\begin{subequations}
	\label{eq:KKT_conditions_inertia}
	\begin{align}
	\intertext{State equation:}
		\begin{pmatrix}
			\boverdot{\overline{u}}
			\\
			\boverdot{\overline{v}}
			\\
			\boverdot{\overline{q}}
		\end{pmatrix}
		&=
		\frac{1}{\lambda_s}
		\begin{pmatrix}
			\overline{w} - \overline{u}
			\\
			\nicefrac{(\overline{w} - \overline{u})}{(\rho\lambda_s)}
			- \nicefrac{\overline{v}}{\rho}
			\\
			(\mathbb{C} + \B ) (\overline{p} - \overline{q})
		\end{pmatrix}
		+
		\begin{pmatrix}
			0
			\\
			\nicefrac{\overline{f}}{\rho}
			\\
			0
		\end{pmatrix}
		\\
		-\div(
			\D\symnabla \overline{w}
			+
			\E \overline{p}
		)
		&=
		\nicefrac{\overline{v}}{\lambda_s}
		+
		\nicefrac{(\overline{w} - \overline{u})}{\lambda_s}
		\\
		\overline{p}
		&=
		R_s(\E \symnabla (\overline{w} - \overline{u}) + \overline{q})
		\\
		(\overline{u}, \overline{v}, \overline{q})(0)
		&=
		(u_0,v_0,q_0)
		\\
		\intertext{Adjoint equation:}
		\begin{pmatrix}
			\boverdot{\varphi}_1
			\\
			\boverdot{\varphi}_2
			\\
			\boverdot{\varphi}_3
		\end{pmatrix}
		&=
		\frac{1}{\lambda_s}
		\begin{pmatrix}
			\eta^* - \varphi_1
			\\
			\nicefrac{(\eta^* - \varphi_1)}{(\rho\lambda_s)}
			-
			\nicefrac{\varphi_2}{\rho}
			\\
			(\mathbb{C} + \B )(r^* - \varphi_3)
		\end{pmatrix}
		-Q\Psi_{z}'(\overline{u}, \overline{v}, \overline{q})
		\label{eq:KKT_conditions_adjoint_evolution_inertia}
		\\
		-\div(
			\D\symnabla \eta^*
			+
			\E r^*
		)
		&=
		\nicefrac{\varphi_2}{\lambda_s}
		+
		\nicefrac{(\eta^* - \varphi_1)}{\lambda_s}
		\\
		r^*
		&=
		R_s'(\E^{\top} \symnabla(\overline{w} - \overline{u}) + \overline{q})^*
		(\E^{\top}\symnabla(\eta^* - \varphi_1) + \varphi_3)
		\\
		(\varphi_1, \varphi_2, \varphi_3)(T)
		&=
		0
		\\
		\intertext{Gradient equation:}
		\scalarproduct{\varphi_2}{g}{L^2(L^2(\O;\Rd))}
		&=
		\alpha\scalarproduct{\overline{f}}{g}{\XXX_c}
		\mediumspace
		\forall g \in \XXX_c.
	\end{align}
	\end{subequations}
	
	In particular,
	if
	$\overline{f}$
	is locally optimal for
	\cref{eq:optimization_problem_r_inertia},
	then
	there exists a unique adjoint state
	$(\varphi, \eta^*) \in H^1(\HH \times H^1_D(\O;\Rd))$
	such that \cref{eq:KKT_conditions_inertia} is fulfilled.
\end{theorem}
\begin{proof}
	At first we proof
	that the assertion holds
	when we exchange
	\cref{eq:KKT_conditions_inertia}
	with
	\begin{align}
	\label{eq:abstract_KKT_conditions_inertia}
	\begin{split}
		Q^{-1}(\boverdot{\overline{u}}, \boverdot{\overline{v}},
		\boverdot{\overline{q}})
		+
		\AA_s(\overline{u}, \overline{v}, \overline{q})
		&=
		R\overline{f},
		\mediumspace
		(\overline{u}, \overline{v}, \overline{q})(0)
		=
		(u_0,v_0,q_0),
		\\
		Q^{-1}\boverdot{\varphi}
		+
		\AA_s'(\overline{u}, \overline{v}, \overline{q})^*\varphi
		&=
		-\Psi_{z}'(\overline{u}, \overline{v}, \overline{q}),
		\mediumspace
		\varphi(T) = 0,
		\\
		\scalarproduct{\varphi_2}{g}{L^2(L^2(\O;\Rd))}
		&=
		\alpha\scalarproduct{\overline{f}}{g}{\XXX_c}
		\mediumspace
		\forall g \in \XXX_c.
	\end{split}
	\end{align}
	To this end,
	let
	$\varphi$
	be the solution
	of the second equation in
	\cref{eq:abstract_KKT_conditions_inertia}
	(which unique existence follows as in
	\cite[Lemma 5.11]{paper_abstract})
	and
	$\eta \define \SS_s'(\overline{f})g \in H^1(\HH)$
	for an arbitrary
	$g \in \XXX_c$,
	then
	\begin{align*}
		\scalarproduct{\varphi_2}{g}{L^2(L^2(\O;\Rd))}
		&=
		\scalarproduct{\varphi}{Rg}{L^2(\HH)}
		=
		\scalarproduct{\varphi}{Q^{-1}\boverdot{\eta}}{L^2(\HH)}
		+
		\scalarproduct{\varphi}
		{\AA'_s(\overline{u}, \overline{v}, \overline{q})\eta}
		{L^2(\HH)}
		\\
		&=
		\scalarproduct{Q^{-1}\boverdot{\varphi}}{\eta}{L^2(\HH)}
		+
		\scalarproduct{\AA'_s(\overline{u}, \overline{v}, \overline{q})^*\varphi}
		{\eta}{L^2(\HH)}
		\\
		&=
		-\scalarproduct{\Psi_{z}'(\SS_s(\overline{f}))}
		{\eta}{L^2(\HH)}
	\end{align*}
	holds for all
	$g \in \XXX_c$,
	which implies the equivalence between
	\cref{eq:Fz_derivative_inertia}
	and the last equation in
	\cref{eq:abstract_KKT_conditions_inertia}.
	Moreover,
	it is well known that if
	$\overline{f}$
	is locally optimal for
	\cref{eq:optimization_problem_inertia_s},
	then
	\cref{eq:Fz_derivative_inertia}
	must hold.
	
	Let us now prove
	the equivalence between
	\cref{eq:abstract_KKT_conditions_inertia}
	and
	\cref{eq:KKT_conditions_inertia}.
	We choose
	$h,\xi \in \HH$
	and denote by
	$\eta^* \in H^1_D(\O;\Rdds)$
	the solution of
	\begin{align}
	\label{eq:local1029102019}
	\begin{split}
		-\div(
			\D \symnabla \eta^*
			+
			\E R_s'(\E^{\top} \symnabla (\TT_{R_s}(h) - h_1) + h_3)^*
			(\E^{\top} \symnabla (\eta^* - \xi_1) + \xi_3))
		)
		=
		\frac{\xi_2}{\lambda_s} + \frac{\xi_1 - \eta^*}{\lambda_s^2}
	\end{split}
	\end{align}
	for all
	$\phi \in H^1_D(\O;\Rd)$
	(the existence of
	$\eta^*$
	follows as in
	\cref{lem:T_Frechet_inertia},
	note that the inequalities in
	\cref{lem:Rs_derivative_estimate_inertia}
	hold also for the adjoint operator).
	Then
	\begin{align*}
		\RR_s'(h)^*\xi =
		\begin{pmatrix}
			\eta^*
			\\
			\frac{1}{\lambda_s} \eta^* - \frac{\xi_1}{\lambda_s}
			\\
			R_s'(\E^{\top} \symnabla(\TT_{R_s}(h) - h_1) + h_3)^*
			(\E^{\top} \symnabla(\eta^* - \xi_1) + \xi_3)
		\end{pmatrix}
	\end{align*}
	holds,
	which can be seen as follows:
	Let
	$g \in \HH$
	and abbreviate
	\begin{align*}
		\eta
		&\define
		\TT_{R_s}'(h)g,
		\quad
		\eta_v
		\define
		\frac{\eta - g_1}{\lambda_s},
		\quad
		\eta_q
		\define
		R_s'(\E^{\top} \symnabla (\TT_{R_s}(h) - h_1) + h_3)
		(\E^{\top} \symnabla (\eta - g_1) + g_3),
		\\
		\eta_v^*
		&\define
		\frac{\eta^* - \xi_1}{\lambda_s},
		\quad
		\eta_q^*
		\define
		R_s'(\E^{\top} \symnabla(\TT_{R_s}(h) - h_1) + h_3)^*
		(\E^{\top} \symnabla(\eta^* - \xi_1) + \xi_3).
	\end{align*}
	Testing
	\cref{eq:TDerivative}
	with
	$\phi = \xi_1 - \eta^*$
	gives
	\begin{align*}
		&
		\scalarproduct{\D \symnabla \eta}
		{\symnabla (\xi_1 - \eta^*)}{L^2(\O;\Rdds)}
		+
		\scalarproduct{\eta_v - g_2}{\eta_v^*}{L^2(\O;\Rd)}
		\\
		&\largespace\largespace
		=
		\scalarproduct{\E\eta_q}{\symnabla (\eta^* - \xi_1)}{L^2(\O;\Rdds)}
		\\
		&\largespace\largespace
		=
		\scalarproduct{\E^{\top} \symnabla (\eta - g_1) + g_3}
		{\eta_q^*}{L^2(\O;\Rdds)}
		-
		\scalarproduct{\eta_q}{\xi_3}{L^2(\O;\Rdds)},
	\end{align*}
	and testing
	\cref{eq:local1029102019}
	with
	$\phi = \eta - g_1$
	yields
	\begin{align*}
		\scalarproduct{\D \symnabla \eta^*}
		{\symnabla (\eta - g_1)}{L^2(\O;\Rdds)}
		+
		\scalarproduct{\xi_2 - \eta_v^*}{\eta_v}{L^2(\O;\Rd)}
		=
		\scalarproduct{\E\eta_q^*}
		{\symnabla (g_1 - \eta)}{L^2(\O;\Rdds)},
	\end{align*}
	thus,
	adding both equations together,
	we arrive at
	\begin{align*}
		&
		\scalarproduct{\D\symnabla\eta}{\symnabla\xi_1}
		{H^1(\O;\Rd)}
		+
		\scalarproduct{\eta_v}{\xi_2}{L^2(\O;\Rdds)}
		+
		\scalarproduct{\eta_q}{\xi_3}{L^2(\O;\Rdds)}
		\\
		&\largespace
		=
		\scalarproduct{\D\symnabla\eta^*}{\symnabla g_1}{H^1(\O;\Rd)}
		+
		\scalarproduct{\eta_v^*}{g_2}{L^2(\O;\Rdds)}
		+
		\scalarproduct{\eta_q^*}{g_3}{L^2(\O;\Rdds)},
	\end{align*}
	which is equivalent to
	\begin{align*}
		\scalarproduct{\RR_s'(h)g}{\xi}{\HH}
		=
		\scalarproduct{g}{\RR_s'(h)^*\xi}{\HH}.
	\end{align*}
	
	Now one only has to use
	the definitions
	of
	$\AA_s$
	and
	$R$
	to obtain the equivalence between
	\cref{eq:KKT_conditions_inertia}
	and
	\cref{eq:abstract_KKT_conditions_inertia}.
\end{proof}

\section{Examples}
\label{sec:examples}

Let us conclude
with examples about
a concrete objective function,
the gradient equation in
\cref{thm:KKT_conditions_inertia}
regarding a concrete control space
and finally a realization of
the maximal monotone operator
$A$
(which will be the
von-Mises flow rule).

\paragraph{Objective function}
Let us consider a
tracking type objective function,
that is,
\begin{align*}
	\Psi(u,v,z)
	=
	\frac{1}{2}
	\norm{(u,v,z) - (u_d,v_d,z_d)}{L^2(\HH)}^2
\end{align*}
with a desired state
$(u_d,v_d,z_d) \in L^2(\HH)$.
Then
\begin{align*}
	\Psi_z(u,v,q)
	=
	\frac{1}{2}
	\norm{(u,v,(\mathbb{C}
	+ \B )^{-1}(\mathbb{C}\symnabla u - q))
	- (u_d,v_d,z_d)}{L^2(\HH)}^2
\end{align*}
and
\begin{align*}
	\Psi_z'(u,v,q) =
	\begin{pmatrix}
		\hat{u}
		\\
		v - v_d
		\\
		(\mathbb{C} + \B )^{-1}(\mathbb{C}\symnabla u - q)) - z_d,
	\end{pmatrix}
\end{align*}
where
$\hat{u}$
is such that
$-\div(
	\D \symnabla (\hat{u} - u + u_d)
	-
	((\mathbb{C} + \B )^{-1}(\mathbb{C}\symnabla u - q) - z_d)
)
=
0$,
hence,
in this example the adjoint equation	
in
\cref{eq:KKT_conditions_inertia}
has to be completed by this equation.
Note that when one uses a
finite element approach to solve
\cref{eq:KKT_conditions_inertia}
numerically,
then one can eliminate
this additional equation after multiplying
\cref{eq:KKT_conditions_adjoint_evolution_inertia}
with a test function,
that is,
taking the
$\HH$-scalar product.
When the
$\HH$-scalar product of
$Q\Psi'_z(u, v, q)$
and a test function
$(\eta_1, \eta_2, \eta_3)$
is evaluated,
the term
$\scalarproduct{\D\symnabla \hat{u}}{\eta_1}{L^2(\O;\Rdds)}$
arises,
then one can use the additional equation to
eliminate
$\hat{u}$
(respectively the equation).

\paragraph{Control space}
Let us consider the space
\begin{align*}
	\XXX_c \define
	\{
		f \in H^1(L^2(\O;\Rd)) \cap L^2(H^1(\O;\Rd))
		:
		f(0) = f(T) = 0
	\}
\end{align*}
with the scalar product
\begin{align*}
	\scalarproduct{f}{g}{\XXX_c}
	=
	\scalarproduct{\boverdot{f}}{\boverdot{g}}{L^2(L^2(\O;\Rd))}
	+ \scalarproduct{\nabla f}{\nabla g}{L^2L^2(\O;\Rdd)},
\end{align*}
see
\cref{ex:control_space_inertia}.
The Gradient equation in
\cref{eq:KKT_conditions_inertia}
then becomes
\begin{align*}
	\alpha
	\scalarproduct{\boverdot{f}}{\boverdot{g}}{L^2(L^2(\O;\Rd))}
	+
	\alpha
	\scalarproduct{\nabla f}{\nabla g}{L^2L^2(\O;\Rdd)}
	=
	\scalarproduct{\varphi_2}{g}{L^2(L^2(\O;\Rd))}
\end{align*}
for all
$g \in \XXX_c$,
which is the weak formulation of
\begin{align*}
	\bdoubleoverdot{f} + \Delta f = -\frac{\varphi_2}{\alpha}.
\end{align*}

\paragraph{Maximal monotone operator $A$}
We consider the case of linear kinematic hardening with the von Mises yield condition,
cf.
\cite{hanReddy} for a detailed description of this model. 
In this case, 
$A$ is the convex subdifferential of the indicator functional $I_{\KK(\Omega)}$
of the following set of admissible stresses
\begin{equation*}
    \KK(\Omega) := \{ \tau \in L^2(\O;\Rdds) : |\tau^D(x)| \leq \gamma \quad \text{f.a.a. } x \in \Omega\},
\end{equation*}
where $\tau^D := \tau - \frac{1}{d} \operatorname{tr}(\tau)I$
is the deviator of
$\tau \in \Rdds$ and $\gamma$
denotes the initial uni-axial yield stress, a given material parameter. 
The domain of $A = \partial I_{\KK(\Omega)}$ is trivially ${\KK(\Omega)}$,
which is nonempty, closed and convex.
For the Yosida approximation of $\partial I_{\KK(\Omega)}$,
one obtains by a straightforward calculation
\begin{equation}\label{eq:proj}
    A_\lambda(\tau) = \frac{1}{\lambda}(\tau - \pi_{\mathcal{K}}(\tau)) 
    = \frac{1}{\lambda} \max \Big\{ 0, 1 - \frac{\gamma}{|\tau^D|} \Big\} \tau^D,
\end{equation}
with the resolvent $R_\lambda = \pi_{\KK(\Omega)}$,
where $\pi_{\KK(\Omega)}$ denotes the projection on ${\KK(\Omega)}$ in $L^2(\O;\Rdds)$, i.e.,
$\pi_\KK(\sigma) := \argmin_{\tau \in \KK(\Omega)} \|\tau - \sigma\|_{L^2(\O;\Rdds)}^2$.
We get in particular
\begin{align*}
	R_\lambda(\tau) = \pi_{\KK(\Omega)}(\tau)
	= \tau - \max \Big\{ 0, 1 - \frac{\gamma}{|\tau^D|} \Big\} \tau^D,
\end{align*}
so that
\cref{eq:resolvent_pointwise_inertia}
is fulfilled.

To smoothen this function let $s > 0$
and
\begin{equation*}
        \operatorname{max}_s : \mathbb{R} \rightarrow \mathbb{R} \largespace r \mapsto 
        \begin{cases}
            \max\{0,r\}, & |r| \geq s, \\     
            \tfrac{1}{4s} (r+s)^2,  & |r| < s,
        \end{cases}                
\end{equation*}
we then set
\begin{align*}
	R_s : \Rdds \rightarrow \Rdds, \largespace \tau \mapsto \tau
	- \textnormal{max}_s \Big{(} 1 - \frac{\gamma}{|\tau^D|} \Big{)} \tau^D.
\end{align*}
One easily checks that $\operatorname{max}_s \in C^1(\mathbb{R})$
and we obtain
\begin{align}
\label{eq:smoothYosidaEstimate}
\begin{split}
	&\norm{\partial I_{\lambda,s}(\tau) - \partial I_\lambda(\tau)}{L^2(\O;\Rdds)}
	\\
	&\mediumspace
	\leq
	\frac{1}{\lambda}
	\Big(
		\int_\O \Big| \max_s \Big( 1 - \frac{\gamma}{|\tau(x)^D|} \Big)
		- \max \Big( 0, 1 - \frac{\gamma}{|\tau(x)^D|} \Big) \Big|^2
		\
		|\tau(x)^D|^2 
	\Big)^{\nicefrac{1}{2}}
	\\
	&\mediumspace
	\leq
	\frac{|\O|\gamma s}{4\lambda(1 - s)}
\end{split}
\end{align}
for all $\tau \in L^2(\O;\Rdds)$. Moreover, it is also easy to verify
that
$R_s : \Rdds \to \Rdds$
is Fr\'{e}chet differentiable with
\begin{equation*}
        	R_s'(\tau)h
	= h - \max_s'\Big(1 - \frac{\gamma }{|\tau^D|}\Big) \frac{\gamma }{|\tau^D|^3}
    	(\tau^D \colon h^D) \tau^D
	-\max_s\Big(1 - \frac{\gamma }{|\tau^D|}\Big) h^D
\end{equation*}
and the following lemma shows that $R_s$ is monotone
and Lipschitz continuous,
thus
\cref{assu:standing_assumptions_for_sec42}
\cref{assu:item:Rs_inertia}
is satisfied.
Note also that $R_s'(\tau)$ is self adjoint for every
$\tau \in \Rdds$,
hence,
$\RR_s'(h)$
is also self adjoint for all $h \in \HH$
(cf. the proof of
\cref{thm:KKT_conditions_inertia}).

\begin{lemma}
	For every $s \in (0,1)$, the mapping $R_s$ is monotone
	and Lipschitz continuous with constant 1.
\end{lemma}
\begin{proof}
	It is well known that,
	since
	$\max_s$
	is continuously differentiable
	and convex,
	\begin{align}
	\label{eq:local011122019}
		\max_s(x) - \max_s(y)
		\geq
		\max_s(y)'(x - y)
	\end{align}
	holds for all
	$x,y \in \R$.
	
	Let
	$\tau, \sigma \in \Rdds$,
	w.l.o.g.
	we can assume that
	$|\sigma^D| \geq |\tau^D| > 0$
	and
	\begin{align*}
		\max_s
		\Big{(}
			1 - \frac{\gamma}{|\sigma^D|}
		\Big{)}
		\geq
		\max_s
		\Big{(}
			1 - \frac{\gamma}{|\tau^D|}
		\Big{)},
	\end{align*}
	then,
	using
	\cref{eq:local011122019}
	with
	$x = 1 - \frac{\gamma}{|\tau^D|}$
	and
	$y = 1 - \frac{\gamma}{|\sigma^D|}$,
	we get
	\begin{align*}
		|\max_s
		&
		\Big{(}
			1 - \frac{\gamma}{|\tau^D|}
		\Big{)}
		\tau^D
		-
		\max_s
		\Big{(}
			1 - \frac{\gamma}{|\sigma^D|}
		\Big{)}
		\sigma^D |
		\\
		&\leq
		(
			\max_s
			\Big{(}
				1 - \frac{\gamma}{|\sigma^D|}
			\Big{)}
			- 
			\max_s
			\Big{(}
				1 - \frac{\gamma}{|\tau^D|}
			\Big{)}
		)
		\ |\tau^D|
		+
		\max_s
		\Big{(}
			1 - \frac{\gamma}{|\sigma^D|}
		\Big{)}
		| \tau^D - \sigma^D |
		\\
		&\leq
		\gamma \max_s'
		\Big(
			1 - \frac{\gamma}{|\sigma^D|}
		\Big)
		| \tau^D | \
		|
			\frac{1}{| \tau^D |}
			-
			\frac{1}{| \sigma^D |}
		|
		+
		\max_s
		\Big{(}
			1 - \frac{\gamma}{|\sigma^D|}
		\Big{)}
		| \tau^D - \sigma^D |,
	\end{align*}
	taking into account that
	\begin{align*}
		\gamma | \tau^D | \
		|
			\frac{1}{| \tau^D |}
			-
			\frac{1}{| \sigma^D |}
		|
		=
		\gamma \Big{|}
			\frac{| \tau^D | - | \sigma^D |}{| \sigma^D |}
		\Big{|}
		\leq
		\frac{\gamma}{| \sigma^D |}
		| \tau^D - \sigma^D |
	\end{align*}
	we obtain
	\begin{align*}
		|
			\max_s&
			\Big{(}
				1 - \frac{\gamma}{|\tau^D|}
			\Big{)}
			\tau^D
			-
			\max_s
			\Big{(}
				1 - \frac{\gamma}{|\sigma^D|}
			\Big{)}
			\sigma^D
		|
		\\
		&\leq
		\Big{(}
			\max_s'
			\Big(
				1 - \frac{\gamma}{|\sigma^D|}
			\Big)
			\frac{\gamma}{| \sigma^D |}
			+
			\max_s
			\Big{(}
				1 - \frac{\gamma}{|\sigma^D|}
			\Big{)}
		\Big{)}
		| \tau^D - \sigma^D |
		\\
		&\leq
		\max_s(1) | \tau - \sigma |
		=
		| \tau - \sigma |,
	\end{align*}
	where we used
	\cref{eq:local011122019}
	again with
	$x = 1$
	and
	$y = 1 - \frac{\gamma}{|\sigma^D|}$,
	and the fact that
	$| \tau^D - \sigma^D | \leq | \tau - \sigma |$.
	This proves
	the Lipschitz continuity of
	$R_s$.
	We also get
	\begin{align*}
		(R_s(\tau) &- R_s(\sigma)) \colon (\tau - \sigma)
		\\
		&=
		| \tau - \sigma |^2
		-
		\Big{(}
			\max_s
			\Big{(}
				1 - \frac{\gamma}{|\tau^D|}
			\Big{)}
			\tau^D
			-
			\max_s
			\Big{(}
				1 - \frac{\gamma}{|\sigma^D|}
			\Big{)}
			\sigma^D
		\Big{)}
		\colon
		(\tau - \sigma)
		\\
		&\geq
		| \tau - \sigma |^2
		-
		| \tau - \sigma |^2
		=
		0
	\end{align*}
	which shows
	the monotonicity of
	$R_s$.
\end{proof}

\subsection*{Acknowledgements}

I would like to thank Christian Meyer
(Technische Universit\"at
Dortmund)
for fruitful discussions regarding
\cref{cor:W1r_existence_inertia}.

\bibliographystyle{jnsao}
\bibliography{references}

\end{document}